\documentclass[a4paper,11pt]{article}
\usepackage[latin1]{inputenc}
\usepackage[english]{babel}
\usepackage{amsmath}
\usepackage{amsfonts}
\usepackage{amssymb}
\usepackage{epsfig}
\usepackage{amsopn}
\usepackage{amsthm}
\usepackage{color}
\usepackage{graphicx}
\usepackage{enumerate}
\usepackage{mathrsfs}
\usepackage{cite}
\newtheorem{theorem}{Theorem}[section]
\newtheorem{corollary}[theorem]{Corollary}
\newtheorem{lemma}[theorem]{Lemma}
\newtheorem{proposition}[theorem]{Proposition}

\newtheorem*{theorem*}{Theorem}
\newtheorem*{lemma*}{Lemma}
\newtheorem*{remark*}{Remark}
\newtheorem*{definition*}{Definition}
\newtheorem*{proposition*}{Proposition}
\newtheorem*{corollary*}{Corollary}
\numberwithin{equation}{section}
\newcommand{\real}{\mathbb{R}}

\let\ced=\c         

\def\qed{\,\unskip\kern 6pt \penalty 500
\raise -2pt\hbox{\vrule \vbox to8pt{\hrule width 6pt
\vfill\hrule}\vrule}\par}
\definecolor{darkblue}{rgb}{0.05, .05, .65}
\definecolor{darkgreen}{rgb}{0.1, .65, .1}
\definecolor{darkred}{rgb}{0.8,0,0}
\newcommand{\beqn}{\begin{equation}}
\newcommand{\eeqn}{\end{equation}}
\newcommand{\bear}{\begin{eqnarray}}
\newcommand{\eear}{\end{eqnarray}}
\newcommand{\bean}{\begin{eqnarray*}}
\newcommand{\eean}{\end{eqnarray*}}
\newcommand{\mbx}{\boldsymbol{\xi}}

\begin{document}

\title{\huge \bf Second order asymptotics and uniqueness for self-similar profiles to a singular diffusion equation with gradient absorption}

\author{
\Large Razvan Gabriel Iagar\,\footnote{Departamento de Matem\'{a}tica
Aplicada, Ciencia e Ingenieria de los Materiales y Tecnologia
Electr\'onica, Universidad Rey Juan Carlos, M\'{o}stoles,
28933, Madrid, Spain, \textit{e-mail:} razvan.iagar@urjc.es}
\\[4pt] \Large Philippe Lauren\ced{c}ot\,\footnote{Laboratoire des Math\'ematiques (LAMA) UMR 5217, Universit\'e Savoie-Mont Blanc, CNRS, F-73000, Chamb\'ery France. \textit{e-mail:} philippe.laurencot@univ-smb.fr}\\ [4pt] }
\date{\today}
\maketitle

\begin{abstract}
Solutions in self-similar form presenting finite time extinction to the singular diffusion equation with gradient absorption
$$
\partial_t u - \mathrm{div}(|\nabla u|^{p-2}\nabla u) +|\nabla u|^{q}=0 \qquad {\rm in} \ (0,\infty)\times\real^N
$$
are studied when $N\geq1$ and the exponents $(p,q)$ satisfy
$$
p_c=\frac{2N}{N+1}<p<2, \qquad p-1<q<\frac{p}{2}.
$$
\emph{Existence and uniqueness} of such a solution are established in dimension $N=1$. In dimension $N\geq2$, existence  of radially symmetric self-similar solutions is proved and a fine description of their \emph{behavior as $|x|\to\infty$} is provided.
\end{abstract}

\bigskip


\smallskip


\section{Introduction}

The aim of this paper is to give an insight on the phenomenon of finite time extinction of non-negative solutions to the following singular diffusion equation with an absorption term depending on the gradient
\begin{equation}\label{eq1}
\partial_t u -\Delta_{p}u+|\nabla u|^{q}=0, \qquad (t,x)\in(0,\infty)\times\real^N,
\end{equation}
where $N\geq1$, $\Delta_p u = \mathrm{div}(|\nabla u|^{p-2}\nabla u)$ is the standard $p$-Laplacian operator, and
\begin{equation}\label{exp.eq1}
p_c:=\frac{2N}{N+1}<p<2, \qquad p-1<q<\frac{p}{2}.
\end{equation}
A significant feature of Eq.~\eqref{eq1}, when $p\in(1,2)$, is that it involves a competition between a singular diffusion term, which corresponds to the \emph{fast $p$-Laplacian equation}, and a nonlinear absorption term in the form of a power of the euclidean norm of the gradient. This competition can generate a number of different mathematical properties of solutions to Eq.~\eqref{eq1}, according to the range of the exponents $(p,q)$, that were classified in the authors' previous work \cite{IL12}. In particular, it is shown therein that, in the specific range~\eqref{exp.eq1}, \emph{finite time extinction} occurs for suitable initial conditions. Recall that a solution $u$ to Eq.~\eqref{eq1} is said to vanish in finite time if there exists $T_e\in(0,\infty)$ such that $u(t)\not\equiv0$ for any $t\in(0,T_e)$, but $u(t,x)=0$ for any $(t,x)\in (T_e,\infty)\times \real^N$.

In the last three decades, both the semilinear problem ($p=2$) and the degenerate diffusion-absorption problem ($p>2$) related to Eq.~\eqref{eq1} have been investigated, with emphasis on the large time behavior. As an outcome of many successive works, see for example \cite{ATU04, BKaL04, BL99, BLS01, BLSS02, BSW02, BRV97, BGK04, GL07, GGK03, Gi05, BVD13, HJBook} and references therein, an almost complete understanding of the qualitative properties of solutions to Eq.~\eqref{eq1} in the semilinear case $p=2$ is available. It has been noticed in particular that the gradient-type absorption becomes very strong and dominant if $q\in(0,1)$, where finite time extinction is a typical phenomenon, while for $q>1$ the diffusion implies that solutions are positive in $\real^N$ and global in time, decaying with some rate as $t\to\infty$. Concerning the degenerate case $p>2$, the situation is very different: indeed, on the one hand, the support of compactly supported solutions advances in time with finite speed and interfaces appear, as a classical effect of the slow diffusion, see \cite{BtL08, Shi04b, NT11}. But on the other hand, it is proved that, if $q\in (1,p-1]$, then the dynamics of Eq.~\eqref{eq1} is fully governed by the absorption term \cite{ILV, LV07}, giving rise to asymptotic profiles with features such as shape and regularity specific to a Hamilton-Jacobi equation instead of a nonlinear diffusion one.

The range $p\in (1,2)$, where the diffusion is no longer degenerate, but becomes singular when $\nabla u$ vanishes, is a very interesting one and has been considered by the authors in a number of works during the last decade. The starting point of this research in the fast diffusion range stems from \cite{IL12}, in which the well-posedness of the Cauchy problem in the sense of viscosity solutions, together with a comparison principle in the spirit of \cite{OS97} and optimal gradient estimates of solutions to Eq.~\eqref{eq1} are established. With the help of these gradient estimates, the ranges of algebraic decay as $t\to\infty$, exponential decay as $t\to\infty$ and finite time extinction are also identified in \cite{IL12}. Restricting ourselves to the supercritical fast diffusion range,
\begin{equation*}
p_c=\frac{2N}{N+1}<p<2,
\end{equation*}
three critical values of the exponent $q$ are uncovered in this analysis, namely:

$\bullet$ $q=q_*:=p-N/(N+1)$ is a critical exponent separating, in the large time behavior, the range $q>q_*$ where the diffusion term rules over the dynamics, and the range $p/2<q<q_*$ where a balance between the singular diffusion and the gradient absorption is achieved, leading to \emph{very singular solutions} in self-similar form, with characteristics inherited from the two terms competing in Eq.~\eqref{eq1}, as asymptotic profiles. The latter is established in \cite{IL14}, following the existence, uniqueness and classification of very singular self-similar solutions obtained by the authors in \cite{IL13b}, see also \cite{Shi04} for the range $q\in (1,q_*)$.

$\bullet$ $q=p/2$ is a critical exponent separating the range $q>p/2$, where solutions are positive in $\real^N$ and present an algebraic time decay as $t\to\infty$, and the range $0<q<p/2$ where finite time extinction takes place (at least for initial conditions rapidly decaying at infinity). This critical exponent is \emph{specific to the fast diffusion range}, since it plays no role at all for $p>2$, and it has been identified in \cite{IL12}. Moreover, the analysis performed in \cite{IL13a} for precisely the critical case $q=p/2$ led to a classification of \emph{eternal} self-similar solutions in exponential form, with a very thin difference, of logarithmic scale, between the fast decay and the slow decay as $|x|\to\infty$ of the self-similar profiles.

$\bullet$ $q=p-1$ is a critical exponent separating two different mechanisms of finite time extinction. Indeed, as shown in \cite{ILS17}, in the range $q\in(0,p-1)$, a rather striking phenomenon, known as \emph{instantaneous shrinking of supports}, takes place genuinely; that is, if $u_0$ is an initial condition with sufficiently fast decay as $|x|\to\infty$ (see \cite{ILS17} for the precise statements), then the solution $u$ to Eq.~\eqref{eq1} with initial condition $u_0$ becomes immediately compactly supported for any $t>0$, and then the support shrinks as $t>0$ increases, leading to a \emph{single point extinction} at $t=T_e$. On the contrary, it is expected that, for $q\in(p-1,p/2)$, the solutions stay positive for any $(t,x)\in(0,T_e)\times\real^N$ and that finite time extinction occurs \emph{simultaneously} as $t\to T_e$. In particular, the critical case $q=p-1$ is studied in \cite{IL17, IL18} and such simultaneous extinction is  proved in this case, along with a precise description of the self-similar behavior at the extinction time. The proof takes advantage of an underlying variational structure which is only available for the specific choice $q=p-1$ in a radially symmetric setting.

We thus notice that there is still a gap remaining in the previous classification, which is related exactly to the range \eqref{exp.eq1}. The authors considered this range in their short note \cite{IL18b} and identified both the optimal tail of the initial data $u_0$ for finite time extinction to take place; that is, there is $C_0>0$ such that
\begin{equation}\label{tail}
u_0(x)\leq C_0 (1+|x|)^{-q/(1-q)}, \qquad x\in\real^N,
\end{equation}
and the extinction rate in the case when a more restrictive decay as $|x|\to\infty$ than \eqref{tail} is fulfilled. More precisely, \cite[Theorem~1.2]{IL18b} establishes that, if $u_0$ is an initial condition decaying as
\begin{equation}\label{tail.fast}
u_0(x)\leq K_0|x|^{-(p-q)/(q-p+1)}, \qquad x\in\real^N,
\end{equation}
for some $K_0>0$, then the solution $u$ to Eq.~\eqref{eq1} with initial condition $u_0$ vanishes in finite time $T_e\in(0,\infty)$ and there are positive constants $0<c_1\le C_1$ and $0<c_\infty\le C_\infty$ such that
\begin{equation}\label{ext.rate}
c_{\infty}(T_e-t)^{\alpha}\leq\|u(t)\|_{\infty}\leq C_{\infty}(T_e-t)^{\alpha}, \qquad t\in(0,T_e),
\end{equation}
and
\begin{equation}\label{ext.rate.L1}
c_1(T_e-t)^{\alpha-N\beta}\leq\|u(t)\|_{1}\leq C_1(T_e-t)^{\alpha-N\beta}, \qquad t\in(0,T_e),
\end{equation}
with
\begin{equation}\label{SSexp}
	\alpha := \frac{p-q}{p-2q} >0, \qquad \beta := \frac{q-p+1}{p-2q}>0.
\end{equation}
Besides, the following optimal gradient estimate 
\begin{equation}\label{grad.est}
\left|\nabla u^{-(q-p+1)/(p-q)}(t,x)\right|\leq K_1\left[ 1+\|u_0\|_{\infty}^{(p-2q)/p(p-q)} t^{-1/p} \right],
\end{equation}
is established in \cite[Theorem~1.3~(iii)]{IL12}:\footnote{The integrability of $u_0$ is also assumed in the statement of \cite[Theorem~1.3~(iii)]{IL12} but is actually not needed for its validity, as the gradient estimate only involves the $L^\infty$-norm of $u_0$} for solutions to Eq. \eqref{eq1} with bounded and continuous initial conditions and it holds true in the positivity set of $u$; that is, for
$$
(t,x)\in\mathcal{P}(u):=\{(t,x)\in (0,T_e)\times\real^N: u(t,x)>0\},
$$
for some constant $K_1>0$ depending only on $p$ and $q$. The extinction rates~\eqref{ext.rate} and~\eqref{ext.rate.L1} strongly suggest that the behavior near extinction in the range~\eqref{exp.eq1} of exponents $(p,q)$ should be a self-similar one. The aim of this paper is then to prove the existence of self-similar solutions, along with some properties of their profiles. We are thus in a position to state our main results.

\medskip

\noindent \textbf{Main results}. Let us consider Eq.~\eqref{eq1} in the range~\eqref{exp.eq1} of exponents $(p,q)$. We look for radially symmetric self-similar solutions to Eq.~\eqref{eq1} presenting finite time extinction, in the form
\begin{equation}\label{SSS}
u(t,x)=(T-t)^{\alpha}f(|x|(T-t)^{\beta}),
\end{equation}
where the self-similar exponents are given by \eqref{SSexp}. Introducing the ansatz~\eqref{SSS} into Eq.~\eqref{eq1} gives the ordinary differential equation solved by the self-similar profile $f$ of a solution in the form~\eqref{SSS}, namely
\begin{equation}\label{SSODE}
(|f'|^{p-2}f')'(r)+\frac{N-1}{r}(|f'|^{p-2}f')(r)+\alpha f(r)+\beta rf'(r)-|f'(r)|^q=0,
\end{equation}
with independent variable $r=|x|(T-t)^{\beta}\ge 0$. In addition, since we expect the self-similar solution $u$ to be smooth, we impose the condition $f'(0)=0$. A formal analysis of~\eqref{SSODE}, by letting
$$
f(r)\sim Cr^{-\theta} \qquad {\rm as} \ r\to\infty,
$$
reveals that positive solutions to \eqref{SSODE} may only have the following two behaviors as $r\to\infty$:
\begin{equation}\label{beh.infty}
f(r)\sim Cr^{-q/(1-q)} \qquad {\rm or} \qquad f(r)\sim Cr^{-(p-q)/(q-p+1)}.
\end{equation}
Noticing that
$$
\frac{p-q}{q-p+1}-\frac{q}{1-q}=\frac{p-2q}{(q-p+1)(1-q)}>0
$$
in the range~\eqref{exp.eq1} of exponents $(p,q)$, we deduce that the fastest decay is the second one in~\eqref{beh.infty} and we will be thus looking for solutions to~\eqref{SSODE} enjoying this decay property. In order to simplify the rest of the exposition, we introduce the following two constants
\begin{equation}\label{const}
\mu:=\frac{p-q}{q-p+1}=\frac{\alpha}{\beta}>N, \qquad K^*:=\frac{1}{\mu}\left[(p-1)(\mu+1)-N+1\right]^{1/(q-p+1)}.
\end{equation}
We specialize now to dimension $N=1$ and we state our first result, which deals with existence and uniqueness of a radially symmetric self-similar solution with fast decay as $r\to\infty$.

\begin{theorem}[Existence and uniqueness, $N=1$]\label{th.dim1}
Let $N=1$. There exists a unique $a^*\in(0,\infty)$ such that the solution $f^*$ to~\eqref{SSODE} with initial conditions $f^*(0)=a^*$ and $(f^*)'(0)=0$ is positive on $(0,\infty)$ and enjoys the fast decay
\begin{equation}\label{fast.decay}
f^*(r)\sim K^*r^{-\mu} \qquad {\rm as} \ r\to\infty,
\end{equation}
where the constants $\mu$ and $K^*$ are defined in~\eqref{const}.
\end{theorem}

We are able to overcome the (usually) very difficult problem of  uniqueness of the self-similar profile by a technique relying on a fine analysis of an auxiliary dynamical system, together with a shifting in space in self-similar variables adapted from \cite{FGV}. The outcome of the former is the identification of the second term in the asymptotic expansion of $f^*(r)$ as $r\to\infty$, see Theorem~\ref{th.dimN} below. The combination of both arguments allows us to prove a monotonicity result among global self-similar solutions with the desired behavior~\eqref{tail.fast}. Unfortunately, it is no longer possible to employ part of this technique in dimension $N\geq2$, in particular the shifting method used in Section~\ref{sec.uniq} in order to prove the uniqueness of the self-similar profile $f^*$ in dimension $N=1$. Still, existence and a fine analysis of the tail of the self-similar profiles are available in general dimensions, as follows.

\begin{theorem}[Existence and tail description, $N\geq 1$]\label{th.dimN}
Let $N\geq 1$. Then there is a closed subset $\mathcal{B}$ of $(0,\infty)$ such that, for any $a\in\mathcal{B}$, the solution $f(\cdot;a)$ to~\eqref{SSODE} with initial conditions $f(0;a)=a$ and $f'(0;a)=0$ is positive on $(0,\infty)$ and satisfies
\begin{equation}\label{tail.SSS}
f(r;a)\sim r^{-\mu}(K^*-Ar^{-\theta}) \;\; {\rm as} \ r\to\infty \;\; \text{ with }\;\; \theta:=\frac{N(p-1)-q(N-1)}{p-1},
\end{equation}
for some $A>0$. In addition, $\mathcal{B}=\{a^*\}$  when $N=1$, with $a^*$ defined in Theorem~\ref{th.dim1}.
\end{theorem}

Notice that $\theta\in(0,1)$ for $N\geq2$, taking into account the range~\eqref{exp.eq1} of exponents $(p,q)$. We thus identify in~\eqref{tail.SSS} a precise decay rate of the profiles up to the second order as $r\to\infty$, which is an essential tool in the proof of the uniqueness of the profile $f^*$ in Theorem~\ref{th.dim1} when $N=1$. We conjecture that the uniqueness of the self-similar profile with decay~\eqref{tail.SSS} also holds true in any dimension $N\geq2$, but some different ideas are to be found for its proof.

The proof of the existence of at least one positive solution $f$ to \eqref{SSODE} which decays as $K^*r^{-\mu}$ as $r\to\infty$ relies on a rather classical shooting method. In contrast, the identification of the second term in the asymptotic expansion as $r\to\infty$ is more involved and it seems to us that the approach we develop here is the most original part of the paper. Specifically, we first transform \eqref{SSODE} into an autonomous quadratic three dimensional system. While such transformations have already been employed and proved useful to study self-similar solutions to the porous medium equation with or without absorption or source terms, see \cite{J53, Hu91, AG95, ISV08, ILS24, ILS24b}, as far as we know it is the first time that it is used for a quasilinear diffusion equation involving a $p$-Laplacian and a gradient term. Once this transformation is performed, establishing \eqref{tail.SSS} amounts to study precisely the behavior of the trajectories of this dynamical system lying on the two-dimensional stable manifold of a specific critical point of it. This analysis requires in particular a rather precise description of the stable manifold. 

\medskip

\noindent \textbf{Organization of the paper}. For the proofs of the main results, we employ a \emph{variety of techniques}, as described now. The existence of self-similar solutions is established in Section~\ref{sec.exist} by a shooting method. In the subsequent Section~\ref{sec.dynamic}, we introduce a transformation mapping the differential equation~\eqref{SSODE} into a three dimensional autonomous dynamical system. A deeper analysis of a specific critical point of this system is performed in order to establish the local behavior~\eqref{tail.SSS} as $r\to\infty$, which completes the proof of Theorem~\ref{th.dimN}. This behavior has independent interest in the analysis of the solutions, but, restricting to dimension $N=1$, it becomes also a fundamental step in the quest for the uniqueness as stated in Theorem~\ref{th.dim1}. Indeed, monotonicity of self-similar solution will be proved in Section~\ref{sec.uniq} using a shifting technique at the level of self-similar profiles, together with a clever use of the comparison principle. Uniqueness then follows from a separation between self-similar solutions stemming from the already established precise behavior at the second order.

\section{Existence of self-similar profiles in dimension $N\geq1$}\label{sec.exist}

This section is devoted to the shooting method leading to the proof of the existence of self-similar solutions with fast decay, which borrows ideas from \cite{IL13b}. For any $a\in(0,\infty)$, we consider the Cauchy problem for the equation~\eqref{SSODE} with initial conditions $f(0)=a$ and $f'(0)=0$. Introducing
$$
F(r):=-(|f'|^{p-2}f')(r),
$$
the problem can be written as
\begin{equation}\label{CP}
\left\{\begin{array}{l}
	f'(r)=-|F(r)|^{(2-p)/(p-1)}F(r),\\
	\\
	F'(r)+\displaystyle{\frac{N-1}{r}} F(r)=\alpha f(r)-\beta r|F(r)|^{(2-p)/(p-1)}F(r)-|F(r)|^{q/(p-1)},\\
	\\
	f(0)=a, \ F(0)=0.
\end{array}\right.
\end{equation}
Since $q>p-1$ and $(2-p)/(p-1)>0$, the right-hand side of \eqref{CP} is locally Lipschitz continuous. There is thus a unique maximal $C^1$-smooth solution $(f(\cdot;a),F(\cdot;a))$ to the system \eqref{CP}, defined on an interval $[0,R_{\max}(a))$ and such that
\begin{equation*}
F'(0;a)=\frac{\alpha a}{N}>0.
\end{equation*}
Moreover, either $R_{\max}(a)=\infty$, or
\begin{equation*}
	R_{\max}(a) < \infty \;\;\text{ and }\;\; \lim_{r\to R_{\max}(a)} \big( |f(r;a)| + |F(r;a)| \big) = \infty.
\end{equation*}
Setting
$$
R(a):=\inf\{r\geq0: f(r;a)=0\}\leq R_{\max}(a),
$$
the positivity of $a$ and the continuity of $f(\cdot;a)$ ensure that $R(a)>0$. Throughout this section, we omit the parameter $a$ in the notation where there is no danger of confusion.

We gather in the following statements a few general properties of $f(\cdot;a)$.

\begin{lemma}\label{lem.1}
Let $a>0$. We have
\begin{subequations}\label{elem}
\begin{equation}\label{decr}
	-(a\alpha)^{1/q}\leq f'(r;a)<0, \qquad r\in(0,R(a)).
\end{equation}
If moreover $R(a)=\infty$, then
\begin{equation}
	\lim\limits_{r\to\infty}f(r;a)=\lim\limits_{r\to\infty}f'(r;a)=0, \label{elem1}
\end{equation}
and there is $\kappa_0>0$ depending only on $p$ and $q$ such that
\begin{equation}
	\left| \left( f^{-1/\mu} \right)'(r;a) \right| \le \kappa_0 \left[ 1 + \|f\|_\infty^{1/(\alpha p)} \right], \qquad r>0, \label{elem2}
\end{equation}
\end{subequations}
\end{lemma}

\begin{proof}
	The proof of~\eqref{decr} and~\eqref{elem1} is exactly the same as that of \cite[Lemma~2.1]{IL13b}, to which we refer. As for~\eqref{elem2}, it follows from~\eqref{SSODE} and~\eqref{elem1} that $f\in W^{1,\infty}([0,\infty))$, while the positivity of $f$ and~\eqref{SSODE} imply that
	\begin{equation*}
		u(t,x) = (1-t)_+^\alpha f\big(|x| (1-t)_+^\beta\big), \qquad (t,x)\in [0,\infty) \times \real^N,
	\end{equation*}
	solves Eq.~\eqref{eq1} with a positive initial condition $u(0)\in W^{1,\infty}(\real^N)$. It then follows from~\eqref{grad.est} that
	\begin{equation*}
		(1-t)^{\beta-\alpha(q-p+1)/(p-q)}\left|\left(f^{-(q-p+1)/(p-q)}\right)'(r)\right|\leq K_1\left[1+\frac{\|f\|_\infty^{1/(\alpha p)}}{t^{1/p}}\right].
	\end{equation*}
	Since
	$$
	\beta-\frac{\alpha(q-p+1)}{p-q}=0,
	$$
	we obtain~\eqref{elem2} with $\kappa_0=K_1$ by letting $t\to 1$ in the previous estimate.
\end{proof}

We next introduce the following energy (which is actually used in the proof of \cite[Lemma~2.1]{IL13b})
\begin{equation}\label{energy}
E(r;a)=\frac{p-1}{p}|f'(r;a)|^p+\frac{\alpha}{2}|f(r;a)|^2, \qquad r\in[0,R_{\max}(a)).
\end{equation}
With the aid of this energy, we show that the solutions to the system~\eqref{CP} cannot have a blow-up at a finite value of $r$.

\begin{lemma}\label{lem.2}
Let $a\in(0,\infty)$. Then $R_{\max}(a)=\infty$.
\end{lemma}

\begin{proof}
We deduce from the definition~\eqref{energy} of $E$ and~\eqref{SSODE} that
\begin{equation*}
\begin{split}
E'(r)&=f'(r)[(p-1)|f'(r)|^{p-2}f''(r)+\alpha f(r)]\\
&=f'(r)\left[-\frac{N-1}{r}|f'(r)|^{p-2}f'(r)-\beta rf'(r)+|f'(r)|^{q}\right]\\
&=-\frac{N-1}{r}|f'(r)|^p-\beta r|f'(r)|^2+|f'(r)|^{q}f'(r)\leq |f'(r)|^{q+1}-\beta r|f'(r)|^2.
\end{split}
\end{equation*}
Observe that we cannot make use of~\eqref{decr}, as it only holds true on $(0,R(a))$, and this is why we cannot directly control the sign of $E'$ on $(0,R_{\max}(a))$. Nevertheless, since $q<p/2<1$ by~\eqref{exp.eq1}, an application of Young's inequality entails that, for $0<r_0\le r < R_{\max}(a)$,
\begin{align*}
	E'(r) & \le (\beta r)^{(q+1)/2} |f'(r)|^{q+1} (\beta r)^{-(q+1)/2} - \beta r|f'(r)|^2 \\
	& \le \frac{q+1}{2} \beta r |f'(r)|^2 + \frac{1-q}{2} (\beta r)^{-(q+1)/(1-q)} - \beta r|f'(r)|^2 \\
	& \le (\beta r_0)^{-(q+1)/(1-q)}.
\end{align*}
Integrating over $(r_0,r)$, we end up with
\begin{equation*}
	0 \le E(r) \le E(r_0) + (\beta r_0)^{-(q+1)/(1-q)} r, \qquad r_0\le r < R_{\max}(a),
\end{equation*}
which prevents the blow-up of both $f$ and $f'$ at a finite value of $r$ and completes the proof.
\end{proof}

Having next in mind that we want to prove the existence of a solution with the decay~\eqref{fast.decay} as $r\to\infty$, we introduce
$$
w(r;a):=r^{\mu}f(r;a), \qquad (r,a)\in[0,\infty)\times(0,\infty),
$$
with $\mu$ defined in \eqref{const}. Straightforward calculations show that $w(\cdot;a)$ solves the differential equation
\begin{equation}\label{ODEw}
\begin{split}
(p-1)r^2w''(r)&+(N-1-2\mu(p-1))rw'(r)+\mu[(p-1)(\mu+1)-N+1]w(r)\\&+|W(r)|^{2-p}[\beta r^{\gamma+1}w'(r)-|W(r)|^q]=0,
\end{split}
\end{equation}
for $r>0$, with
$$
W(r;a):=rw'(r;a)-\mu w(r;a), \qquad \gamma:=\frac{2q-p}{q-p+1}=-\frac{1}{\beta}<0.
$$
We next split the range of $a\in(0,\infty)$ into three disjoint sets according to the expected properties of $w(\cdot;a)$:
\begin{equation*}
\begin{split}
&\mathcal{A}:=\{a>0: {\rm there \ exists} \ R_1(a)\in(0,R(a)) \ {\rm such \ that} \ w'(R_1(a);a)=0\},\\
&\mathcal{B}:=\{a>0: w'(\cdot;a)>0 \ {\rm in} \ (0,\infty) \ {\rm and} \ \lim\limits_{r\to\infty}w(r;a)<\infty\},\\
&\mathcal{C}:=\{a>0: w'(\cdot;a)>0 \ {\rm in} \ (0,\infty) \ {\rm and} \ \lim\limits_{r\to\infty}w(r;a)=\infty\}.
\end{split}
\end{equation*}
Recall that
$$
w'(r;a)=r^{\mu-1}[rf'(r;a)+\mu f(r;a)]\sim\mu ar^{\mu-1}>0
$$
as $r\to0$, so that $w'(\cdot;a)>0$ in a right neighborhood of $r=0$. Therefore, $\mathcal{A}\cup\mathcal{B}\cup\mathcal{C}=(0,\infty)$. In the next subsections, we perform a careful analysis of these three sets.

\subsection{Characterization of the set $\mathcal{A}$}\label{subsec.A}

We begin with a lemma listing some general properties of the solutions $w(\cdot;a)$ to~\eqref{ODEw} for $a\in\mathcal{A}$.

\begin{lemma}\label{lem.A1}
Let $a>0$. Then the following statements are equivalent:

(i) $a\in\mathcal{A}$.

(ii) There exists $R_1(a)\in(0,R(a))$ such that $w'(R_1(a);a)=0$, $w'(\cdot;a)>0$ in $(0,R_1(a))$, $w'(\cdot;a)<0$ in $(R_1(a),R(a))$ and $w''(R_1(a);a)<0$.

(iii) $\sup\limits_{[0,R(a))}w(\cdot;a)<K^*$, where $K^*$ is the constant defined in~\eqref{const}.
\end{lemma}

\begin{proof}
Consider $a\in\mathcal{A}$ and denote the smallest positive zero of $w'$ in $(0,R(a))$ by $R_1(a)$, its existence being guaranteed by the definition of $\mathcal{A}$. Then $w'>0$ in $(0,R_1(a))$ and $w''(R_1(a))\leq0$. Assume for contradiction that $w''(R_1(a))=0$. It then follows from evaluating~\eqref{ODEw} at $r=R_1(a)$ that
$$
\mu(\mu K^*)^{q-p+1}w(R_1(a))-(\mu w(R_1(a)))^{q-p+2}=0;
$$
that is, $w(R_1(a))=K^*$. Since $w'(R_1(a))=0$ and the constant function $K^*$ is a solution to~\eqref{ODEw} on $(0,R(a))$, we conclude by uniqueness that $w\equiv K^*$ on $(0,R(a))$, which contradicts the fact that $w(0)=0$. Therefore $w''(R_1(a))<0$ and there is a maximal interval $(R_1(a),R_2)\subseteq(R_1(a),R(a))$ such that $w'<0$ on $(R_1(a),R_2)$. Let us first notice that, by evaluating~\eqref{ODEw} at $r=R_1(a)$, we have
\begin{equation}\label{interm2}
(p-1)R_1(a)^2w''(R_1(a))=(\mu w(R_1(a)))^{q-p+2}-\mu(\mu K^*)^{q-p+1}w(R_1(a)),
\end{equation}
and the negativity of $w''(R_1(a))$, along with the positivity of $q-p+1$, entails that $w(R_1(a))<K^*$. Since the maximum of $w$ on $[0,R_2]$ is attained at $r=R_1(a)$, we further deduce that
\begin{equation}\label{max}
w(r)<K^* \qquad {\rm for \ any} \ r\in[0,R_2].
\end{equation}
Assume now for contradiction that $R_2<R(a)$. It follows that
$$
w'(R_2)=0, \qquad w''(R_2)\geq0,
$$
whence, by evaluating~\eqref{ODEw} at $r=R_2$, we obtain from the similar equality to~\eqref{interm2} but with $R_1(a)$ replaced by $R_2$, that $w(R_2)\geq K^*$, which is a contradiction. Thus, $R_2=R_1(a)$ and we proved that~(i) implies~(ii).

Since~(ii) implies as above that~\eqref{max} holds true on $(0,R(a))$, we readily obtain that~(ii) implies~(iii).

Finally, if $w$ satisfies~(iii), then let us assume for contradiction that $w'$ does not vanish in $(0,R(a))$. Then $w'>0$ in $(0,R(a))$ and $w>0$ in $(0,R(a))$ as well, which gives that $R(a)=\infty$ and there exists a limit
\begin{equation}\label{interm3}
\lim\limits_{r\to\infty}w(r)=l\in(0,K^*).
\end{equation}
This implies in particular that there exists an increasing sequence $(r_k)_{k\geq1}$ such that $r_kw'(r_k)\to0$ as $k\to\infty$. Thus, according to \cite[Lemma~2.9]{IL13b}, there exists a sequence $(\varrho_k)_{k\geq1}$ such that, at the same time,
\begin{equation}\label{interm4}
\lim\limits_{k\to\infty}\varrho_k=\infty, \qquad \lim\limits_{k\to\infty}\varrho_kw'(\varrho_k)=0, \qquad
\lim\limits_{k\to\infty}\varrho_k^2w''(\varrho_k)=0.
\end{equation}
Taking $r=\varrho_k$ in~\eqref{ODEw} and passing to the limit as $k\to\infty$, we infer from~\eqref{interm3} and~\eqref{interm4} that
$$
\mu l(\mu K^*)^{q-p+1}-(\mu l)^{q-p+2}=0,
$$
hence $l=K^*$, which contradicts~\eqref{interm4}. Consequently, $w'$ vanishes at least once in $(0,R(a))$ and $a\in\mathcal{A}$, showing thus that~(iii) implies~(i).
\end{proof}

The next lemma proves that the solutions to~\eqref{CP} corresponding to elements in $\mathcal{A}$ have a compact positivity set.

\begin{lemma}\label{lem.A2}
Let $a>0$. Then $a\in\mathcal{A}$ if and only if $R(a)<\infty$.
\end{lemma}

\begin{proof}
The converse statement is almost obvious: if $R(a)<\infty$, then $f(R(a))=0$ and thus $w(0)=w(R(a))=0$. By Rolle's theorem, $w'$ vanishes at least once in $(0,R(a))$, proving that $a\in\mathcal{A}$.

Consider now $a\in\mathcal{A}$ and assume for contradiction that $R(a)=\infty$. We deduce from Lemma~\ref{lem.A1}~(ii) that $w$ decreases on $(R_1(a),\infty)$ and there exists the limit
\begin{equation}\label{lim}
\lim\limits_{r\to\infty}w(r)=l\in[0,K^*).
\end{equation}
Arguing as in the final part of the proof of Lemma~\ref{lem.A1} above, we conclude that $l=0$. Furthermore, according to Lemma~\ref{lem.1},
\begin{equation*}
\left|\left(f^{-1/\mu}\right)'(r)\right|\leq \kappa_1 := \kappa_0 \left[ 1 + \|f\|_\infty^{1/(\alpha p)} \right], \qquad r\geq 0;
\end{equation*}
hence
$$
f^{-1/\mu}(r)\leq a^{-1/\mu} + \kappa_1 r, \qquad r\geq 0,
$$
which is equivalent, taking into account the positivity of $\alpha$ and $\beta$, to
\begin{equation}\label{interm5}
f(r)\geq\left[a^{-1/\mu} + \kappa_1 r\right]^{-\mu}, \qquad r\geq0.
\end{equation}
The estimate~\eqref{interm5} translates in terms of $w$ as follows:
\begin{equation}\label{interm6}
w(r)\geq\left[\frac{r}{a^{-1/\mu} + \kappa_1 r}\right]^{\mu},
\end{equation}
and we deduce by letting $r\to\infty$ in~\eqref{interm6} with the help of~\eqref{lim} that
$$
0=\lim\limits_{r\to\infty}w(r)\geq \kappa_1^{-\mu}>0,
$$
and a contradiction. Therefore, $R(a)<\infty$ and the proof is complete.
\end{proof}
We end this section with the non-emptiness of the set $\mathcal{A}$.

\begin{lemma}\label{lem.A3}
The set $\mathcal{A}$ is non-empty, open, and it contains an interval $(a^*,\infty)$ for some $a^*>0$.
\end{lemma}

\begin{proof}
We argue as in \cite[Theorem~2]{Shi04} and \cite[Proposition~2.11]{IL13b} by employing a scaling argument. Specifically, we define the function $g(\cdot;a)$ by
\begin{equation}\label{resc1}
f(r;a)=ag(s;a), \qquad s=ra^{(2-p)/p}, \qquad r\in[0,R(a)).
\end{equation}
Plugging~\eqref{resc1} into the equation~\eqref{SSODE}, we obtain after direct calculations that the function $g(\cdot;a)$ solves the following differential equation
\begin{equation}\label{diff.g}
\begin{split}
(|g'|^{p-2}g')'(s;a)+\frac{N-1}{s}(|g'|^{p-2}g')(s;a)&+\alpha g(s;a)+\beta sg'(s;a)\\&-a^{(2q-p)/p}|g'(s;a)|^p=0,
\end{split}
\end{equation}
for $s\in(0,a^{(2-p)/p}R(a))$, with initial conditions
\begin{equation}\label{init.g}
g(0;a)=1, \qquad g'(0;a)=0.
\end{equation}
The limit of~\eqref{diff.g}-\eqref{init.g} as $a\to\infty$ reads
\begin{equation*}
\left\{\begin{array}{ll}
	(|h'|^{p-2}h')'(s)+ \displaystyle{\frac{N-1}{s}} (|h'|^{p-2}h')(s)+\alpha h(s)+\beta sh'(s)=0, & s>0,\\
	& \\
	h(0)=1, \ h'(0)=0, &
\end{array}\right.
\end{equation*}
and it follows from \cite[Proposition~2.11]{IL13b} that there exists $S_0>0$ such that $h(S_0)=0$, $h(s)>0$ for $s\in(0,S_0)$ and $h'(s)<0$ for $s\in(0,S_0]$. It is then easy to deduce from the continuous dependence of solutions to the Cauchy problem \eqref{diff.g}-\eqref{init.g} that $g(\cdot;a)$ vanishes at some positive $s$ depending on $a$ for $a>0$ large enough. Lemma~\ref{lem.A2} therefore provides the existence of $a^*>0$ such that $(a^*,\infty)\subseteq\mathcal{A}$. Furthermore, continuous dependence and Lemma~\ref{lem.A1}~(iii) entail that $\mathcal{A}$ is open.
\end{proof}

\subsection{Characterization of the sets $\mathcal{C}$ and $\mathcal{B}$. Existence}\label{subsec.C}

We next establish some properties of the elements in the set $\mathcal{C}$. In this direction, the constant $K^*$ introduced in \eqref{const} plays a fundamental role.

\begin{lemma}\label{lem.C1}
Let $a>0$. Then $a\in\mathcal{C}$ if and only if
\begin{equation}\label{interm7}
\sup\limits_{r\in[0,R(a))}w(r;a)>K^*.
\end{equation}
\end{lemma}

\begin{proof}
The direct implication follows obviously from the definition of the set $\mathcal{C}$. Conversely, Lemma~\ref{lem.A1} and~\eqref{interm7} imply that $a\notin\mathcal{A}$. Consequently, $w(\cdot;a)$ is increasing, $R(a)=\infty$ and there is $l\in(K^*,\infty]$ such that $w(r;a)\to l$ as $r\to\infty$. If $l<\infty$, the same argument as in the proof of Lemma~\ref{lem.A1} entails that $l=K^*$ and a contradiction. Therefore, $l=\infty$ and $a\in\mathcal{C}$.
\end{proof}

\begin{lemma}\label{lem.C2}
The set $\mathcal{C}$ is non-empty, open and contains an interval of the form $(0,a_*)$ for some $a_*\in(0,\infty)$.
\end{lemma}

\begin{proof}
We argue as in the proof of \cite[Proposition~2.13]{IL13b}. Let $a>0$. By integrating the estimate~\eqref{decr} on $(0,r)$, we infer that
\begin{equation}\label{interm8}
f(r;a)\geq a-(a\alpha)^{1/q}r, \qquad r\in[0,R(a)],
\end{equation}
which in particular gives
$$
R(a)\geq R_0(a):=a^{(q-1)/q}\alpha^{-1/q}.
$$
Consequently, $R_0(a)/2\in(0,R(a))$ and it follows by evaluating~\eqref{interm8} at $r=R_0(a)/2$ that ,
$$
w\left(\frac{R_0(a)}{2};a\right)\geq\left(\frac{a^{(q-1)/q}\alpha^{-1/q}}{2}\right)^{\mu}\frac{a}{2}\geq2^{-(\mu+1)}\alpha^{-\mu/q}a^{(2q-p)/q(q-p+1)}.
$$
Since
$$
\frac{2q-p}{q(q-p+1)}<0,
$$
it follows that $w(R_0(a)/2;a)>K^*$ if $a$ is small enough. Combining the latter with Lemma~\ref{lem.C1} ensures that $\mathcal{C}$ is non-empty and contains a right neighborhood of $a=0$. Moreover, Lemma~\ref{lem.C1} and the continuous dependence ensure that $\mathcal{C}$ is open, completing the proof.
\end{proof}

The proof of the existence part in the statements of Theorems~\ref{th.dim1} and~\ref{th.dimN} is completed by the following result.

\begin{corollary}\label{cor.B}
(a) $\mathcal{B}$ is non-empty and $\mathcal{B}\subseteq[a_*,a^*]$. Moreover, $a\in\mathcal{B}$ if and only if
$$
R(a)=\infty \;\;\text{ and }\;\; \lim\limits_{r\to\infty} w(r;a) = \sup\limits_{r\in(0,\infty)} w(r;a)= K^*.
$$

(b) If $J$ is a connected component of $\mathcal{A}$, then $\inf\,J\ge a_*>0$ and $\inf\,J\in\mathcal{B}$.

(c) If $J$ is a connected component of $\mathcal{C}$, then $\sup\,J\le a^*$ and $\sup\,J\in\mathcal{B}$.
\end{corollary}

\begin{proof}
(a) This follows immediately from Lemmas~\ref{lem.A1}, \ref{lem.A3}, \ref{lem.C1} and~\ref{lem.C2}.

(b) Let $J$ be a connected component of $\mathcal{A}$. Then $\inf\,J\geq a_*>0$ by Lemma~\ref{lem.C2}, while $\inf\,J\notin\mathcal{A}$ since $\mathcal{A}$ is an open set according to Lemma~\ref{lem.A3}. Moreover, for any $a\in J$ and any $r\in[0,R(a))$, we have $w(r;a)<K^*$ by Lemma~\ref{lem.A1}~(iii), so that continuous dependence entails that
$$
w(r;\inf\,J)\leq K^* \qquad {\rm for} \ r\in[0,\infty),
$$
recalling that $R(\inf\,J)=\infty$ by Lemma~\ref{lem.A2}, since $\inf\,J\notin\mathcal{A}$. Consequently, Lemma~\ref{lem.C1} implies that $\inf\,J\notin\mathcal{C}$ and thus $\inf\,J\in\mathcal{B}$, as claimed.

(c) The proof is similar to that of (b) and we omit it here.
\end{proof}

We close this section with a technical result which gives further properties of $w(\cdot;a)$ and $f(\cdot;a)$ for $a\in\mathcal{B}$. This result will be very useful in the forthcoming sections.

\begin{lemma}\label{lem.B1}
Let $a\in\mathcal{B}$. Then
\begin{equation}\label{deriv.w}
\lim\limits_{r\to\infty}rw'(r;a)=0
\end{equation}
and
\begin{equation}\label{deriv.f}
\lim\limits_{r\to\infty}r^{\mu+1}f'(r;a)=-\mu K^*.
\end{equation}
\end{lemma}

\begin{proof}
On the one hand, according to the definition of $\mathcal{B}$, we have
\begin{equation}\label{interm12}
rw'(r;a)=\mu w(r;a)+r^{\mu+1}f'(r;a)\leq\mu K^*+r^{\mu+1}|f'(r;a)|, \qquad r\geq 0.
\end{equation}
On the other hand, combining the gradient estimate~\eqref{elem2} with Corollary~\ref{cor.B} gives, for $r>0$,
\begin{align*}
	|f'(r)| & \le \mu \kappa_0 \left( 1 + \|f\|_\infty^{1/(\alpha p)} \right) f(r)^{(\mu+1)/\mu} \\
	& \le \mu \kappa_0 \left( 1 + \|f\|_\infty^{1/(\alpha p)} \right) (K^*)^{(\mu+1)/\mu} r^{-1-\mu},
\end{align*}
and we further find by replacing the last estimate into~\eqref{interm12} that
\begin{equation*}
rw'(r;a)\leq\mu K^* + \mu \kappa_0 \left( 1 + \|f\|_\infty^{1/(\alpha p)} \right) (K^*)^{(\mu+1)/\mu}.
\end{equation*}
Consequently,
\begin{equation}
	L := \limsup\limits_{r\to\infty} rw'(r) \in [0,\infty), \label{p01}
\end{equation}
while the finiteness of the limit of $w(r)$ as $r\to\infty$ stated in Corollary~\ref{cor.B} implies that
\begin{equation}
	\liminf\limits_{r\to 0} r w'(r) = 0. \label{p02}
\end{equation}
Assume now for contradiction that $L>0$. Setting $z(r):=rw'(r)$ for $r\geq 0$, there is an increasing sequence $(\bar{r}_k)_{k\ge 1}$ such that $\bar{r}_k\to\infty$ as $k\to\infty$ and $z(\bar{r}_k)=L/2$ for all $k\ge 1$. Rolle's theorem then entails the existence of an increasing sequence $(r_k)_{k\geq1}$ such that $r_k\to\infty$ as $k\to\infty$, $z(r_k)\to L$ as $k\to\infty$, and $z'(r_k)=0$ for all $k\ge 0$. Since $r^2w''(r)=rz'(r)-z(r)$, one has $r_k^2w''(r_k)=-z(r_k)$ and it follows by evaluating~\eqref{ODEw} at $r=r_k$ that
\begin{align*}
[N-1-(2\mu+1)(p-1)]z(r_k)&+\mu(\mu K^*)^{q-p+1}w(r_k)\\
&+|W(r_k)|^{2-p}\left[\beta r_k^{\gamma}z(r_k)-|W(r_k)|^q\right]=0.
\end{align*}
Letting $k\to\infty$ and taking into account that $a\in\mathcal{B}$ and $\gamma<0$, we conclude that
\begin{equation}\label{interm14}
[N-1-(2\mu+1)(p-1)]L+(\mu K^*)^{q-p+2}-|L-\mu K^*|^{q-p+2}=0.
\end{equation}
Now, introducing the function,
$$
\varphi(x):=|x-\mu K^*|^{q-p+2}-(\mu K^*)^{q-p+2}+[(2\mu+1)(p-1)-N+1]x, \qquad x\geq 0,
$$
and recalling that $q-p+2>1$, we notice that $\varphi$ is a convex function, so that $\varphi'$ is increasing. Moreover, since $\mu>N\ge 1$,
\begin{align*}
\varphi'(0)&=-(q-p+2)(\mu K^*)^{q-p+1}+\mu(p-1)+(\mu+1)(p-1)-N+1\\
&=\mu(p-1)-(q-p+2)[(\mu+1)(p-1)-N+1]+(\mu+1)(p-1)-N+1\\&=\mu(p-1)-(q-p+1)[(\mu+1)(p-1)-N+1]\\
&=(p-1)[\mu-(\mu+1)(q-p+1)]+(q-p+1)(N-1)\\&=(p-1)(\mu-1)+(q-p+1)(N-1)>0.
\end{align*}
Therefore, $\varphi'(x)\geq0$ for any $x\geq0$, hence $\varphi$ is increasing on $[0,\infty)$ with $\varphi(0)=0$ and this shows that the only solution to \eqref{interm14} is $L=0$, and a contradiction. Therefore, $L=0$ and~\eqref{deriv.w} follows from~\eqref{p01} and~\eqref{p02}.

\medskip

In order to prove now~\eqref{deriv.f}, we recall that
$$
rw'(r)=r(r^{\mu}f(r))'=\mu w(r)+r^{\mu+1}f'(r)
$$
and we immediately deduce \eqref{deriv.f} by passing to the limit as $r\to\infty$, taking into account \eqref{deriv.w} and the fact that $a\in\mathcal{B}$.
\end{proof}

\section{An auxiliary dynamical system. Refined behavior as $r\to\infty$}\label{sec.dynamic}

Once established the existence of elements in $\mathcal{B}$, and thus, of self-similar solutions to Eq.~\eqref{eq1} with the fast decay~\eqref{fast.decay} as $r\to\infty$, the aim of this longer and rather technical section is to compute the second order of their behavior as $r\to\infty$. This step is of independent interest and leads to the decay rate~\eqref{tail.SSS} as $r\to\infty$, thereby completing the proof of Theorem~\ref{th.dimN}. It also comes decisively into play in the proof of the uniqueness part in Theorem~\ref{th.dim1}. The proof is based on a fine analysis of an auxiliary dynamical system obtained from~\eqref{SSODE} by performing a suitable transformation, see Section~\ref{subsec.dynamic} below. But for the time being, we begin with a formal deduction of the expansion~\eqref{tail.SSS}.

\subsection{A formal deduction of the next order}\label{subsec.formal}

Let us recall here that, for $a\in\mathcal{B}$, we have proved in Corollary~\ref{cor.B} that $0<w(r;a)<K^*$ for $r\in(0,\infty)$ and $\lim\limits_{r\to\infty}w(r;a)=K^*$. We aim at finding the second order in this expansion for $a\in\mathcal{B}$, and thus insert the ansatz
\begin{equation}\label{sec.ord}
w(r;a)\sim K^*-Ar^{-\theta}, \qquad {\rm as} \ r\to\infty
\end{equation}
with yet undetermined exponent $\theta$ and constant $A>0$, into the equation~\eqref{ODEw} solved by $w(\cdot;a)$. As we are at a formal level, it is also expected to have
$$
rw'(r;a)\sim\theta Ar^{-\theta}, \qquad r^2w''(r;a)\sim-\theta(\theta+1)Ar^{-\theta}, \qquad {\rm as} \ r\to\infty.
$$
We thus substitute these expansions in~\eqref{ODEw} and obtain
\begin{align*}
o(r^{-\theta})&=-(p-1)\theta(\theta+1)Ar^{-\theta}+(N-1-2\mu(p-1))\theta Ar^{-\theta}+ (\mu K^*)^{q-p+2}\\
& \quad -A\mu[(p-1)(\mu+1)-(N-1)]r^{-\theta}+|\theta Ar^{-\theta}-\mu K^*+\mu Ar^{-\theta}|^{2-p}\\
&\quad\qquad \times\left[\beta r^{\gamma}\theta Ar^{-\theta}-|\theta Ar^{-\theta}-\mu K^*+\mu Ar^{-\theta}|^{q}\right].
\end{align*}
Introducing
\begin{equation*}
	B := (\theta+\mu)(N-1)-2\mu(p-1)\theta-(p-1)\theta(\theta+1)-\mu(p-1)(\mu+1)
\end{equation*}
and using the fact that $\gamma=(2q-p)/(q-p+1)<0$ in order to notice that $r^{-\theta+\gamma}$ is a smaller order term, we further obtain
\begin{align*}
o(r^{-\theta}) &=ABr^{-\theta} +(\mu K^*)^{q-p+2}\\
& \quad +(\mu K^*)^{q-p+2}\left(1-\frac{A(\theta+\mu)}{\mu K^*}r^{-\theta}\right)^{2-p}\left[\frac{\beta\theta A}{(\mu K^*)^q}r^{-\gamma-\theta}-\left(1-\frac{A(\theta+\mu)}{\mu K^*}r^{-\theta}\right)^q\right]\\
&=ABr^{-\theta} +(\mu K^*)^{q-p+2}\\
&\quad +(\mu K^*)^{q-p+2}\left(1-\frac{A(\theta+\mu)(2-p)}{\mu K^*}r^{-\theta}\right)\left(-1+\frac{A(\theta+\mu)q}{\mu K^*}r^{-\theta}\right)\\
&=AB r^{-\theta} +(\mu K^*)^{q-p+2}\left[1-1+\frac{A(q-p+2)(\theta+\mu)}{\mu K^*}r^{-\theta}\right],
\end{align*}
Therefore, the constant $A>0$ can be arbitrary, while $\theta$ should solve the algebraic equation
\begin{equation*}
\begin{split}
(\theta+\mu)(N-1)&-2\mu(p-1)\theta-(p-1)\theta(\theta+1)-\mu(p-1)(\mu+1)\\
&+(\mu K^*)^{q-p+1}(q-p+2)(\theta+\mu)=0.
\end{split}
\end{equation*}
Since
\begin{equation*}
	-2\mu(p-1)\theta-(p-1)\theta(\theta+1)-\mu(p-1)(\mu+1) = -(p-1) (\theta+\mu) (\theta+\mu+1)
\end{equation*}
and $\theta$ is expected to be positive, we may factor out $\theta+\mu$ in the previous identity. Then,
\begin{equation*}
	N-1 - (p-1)(\theta+\mu+1) +(\mu K^*)^{q-p+1}(q-p+2) = 0,
\end{equation*}
and, taking into account that $(\mu K^*)^{q-p+1}=(p-1)(\mu+1)-(N-1)$ by~\eqref{const}, we obtain
\begin{equation*}
	\theta = \frac{N(p-1)- q(N-1)}{p-1}
\end{equation*}
as stated in Theorem~\ref{th.dimN}. As a final remark, it is also interesting to notice that $\theta=1$ if $N=1$ and $\theta\in(0,1)$ if $N\geq2$ due to~\eqref{exp.eq1}.

\subsection{A dynamical system}\label{subsec.dynamic}

In order to establish the expansion~\eqref{sec.ord} (and thus~\eqref{tail.SSS} as an immediate consequence) in a rigorous way, we need to work with an auxiliary dynamical system. We thus go back to $f(\cdot;a)$ with $a\in\mathcal{B}$ and introduce the following transformation, inspired partially by the one used in \cite[Section~4.2]{ISV08}: set $\eta(r)=\ln\,r\in\real$ as a new independent variable and
\begin{equation}\label{change.var}
\left\{\begin{array}{l}
	X(\cdot;a)\circ\eta(r) :=-r(|f'|^{-p}f'f)(r;a),\\
	\\
	Y(\cdot;a)\circ \eta(r) :=-r^2(|f'|^{1-p}f')(r;a),\\
	 \\
	 Z(\cdot;a)\circ \eta(r) :=r^{(p-2q)/(2-p)}|Y(\cdot;a)\circ \eta(r)|^{(q-1)/(2-p)} Y(\cdot;a)\circ\eta(r).
\end{array}\right.
\end{equation}
Observe that the positivity of $f$ and the non-positivity of $f'$ provided by Lemma~\ref{lem.1}, along with~\eqref{change.var}, entail that
\begin{equation}
	\big(X(\eta(r);a),Y(\eta(r);a),Z(\eta(r);a)\big)\in [0,\infty)^3, \qquad r\in (0,\infty). \label{p04}
\end{equation}

\begin{lemma}\label{lem.dynamic}
For $a\in \mathcal{B}$, the functions $\big(X(\cdot;a)\circ\eta,Y(\cdot;a)\circ\eta,Z(\cdot;a)\circ\eta \big)$ solve the following autonomous dynamical system
\begin{equation}\label{syst}
\left\{\begin{array}{l}
\dot{X}=NX-Y-\alpha X^2+\beta XY+XZ,\\[1mm]
\\
\displaystyle{\dot{Y}=\left(2-\frac{(2-p)(N-1)}{p-1}\right)Y+\frac{2-p}{p-1}(\alpha X-\beta Y-Z)Y,}\\[1mm]
\\
\displaystyle{\dot{Z} =\frac{q-p+1}{p-1}Z(Z_*-Z)+\frac{q-p+1}{p-1}(\alpha X-\beta Y)Z,}
\end{array}\right.
\end{equation}
with
\begin{equation}\label{zstar}
	Z_*:=\frac{p-1}{q-p+1}-N+1=\frac{N(p-1)-q(N-1)}{q-p+1}>0,
\end{equation}
where the dot denotes the derivative with respect to $\eta$ and we have dropped the explicit dependence on $\eta$ and $a$ for simplicity.
\end{lemma}

\begin{proof}
The proof is performed by direct calculations. However, since these calculations are a bit tedious, we include them in some detail below for the reader's convenience. We start from the definition of $Y\circ\eta$ and find that
\begin{equation}\label{interm9}
\begin{split}
&|f'(r)|=r^{-2/(2-p)}|(Y\circ\eta)(r)|^{1/(2-p)}, \\ &f'(r)=-r^{-2/(2-p)}|(Y\circ\eta)(r)|^{(p-1)/(2-p)}(Y\circ\eta)(r).
\end{split}
\end{equation}
Then
\begin{equation}\label{interm10bis}
(|f'|^{p-2}f')(r)=-r^{-2(p-1)/(2-p)}|(Y\circ\eta)(r)|^{(2p-3)/(2-p)}(Y\circ\eta)(r)
\end{equation}
and differentiating once more with respect to $r$ gives, taking into account that $\eta'(r)=1/r$,
\begin{equation}\label{interm10}
\begin{split}
(|f'|^{p-2}f')'(r)&=\frac{2(p-1)}{2-p}r^{-p/(2-p)}|(Y\circ\eta)(r)|^{(2p-3)/(2-p)}(Y\circ\eta)(r)\\
&-\frac{p-1}{2-p}r^{-p/(2-p)}|(Y\circ\eta)(r)|^{(2p-3)/(2-p)}(\dot{Y}\circ\eta)(r).
\end{split}
\end{equation}
Next, on the one hand, we replace the terms involving $f'$ in~\eqref{SSODE} with their formulas given in~\eqref{interm9}, \eqref{interm10bis} and~\eqref{interm10}, and on the other hand, we separate in the left-hand side the term featuring $\dot{Y}$. After some direct calculations, we find
\begin{equation}\label{interm11}
\begin{split}
\frac{p-1}{2-p}\dot{Y}(\eta)&=\left(\frac{2(p-1)}{2-p}-(N-1)\right)Y(\eta)+\alpha r^{p/(2-p)}f(r)|Y(\eta)|^{-(2p-3)/(2-p)}\\
&-\beta|Y(\eta)|Y(\eta)-r^{(p-2q)/(2-p)}|Y(\eta)|^{(q-2p+3)/(2-p)},
\end{split}
\end{equation}
where we have employed the notation $Y(\eta):=Y\circ\eta$. Recalling now the definition of $X(\eta)$ and $Z(\eta)$ (with the same convention of notation) from \eqref{change.var}, and taking into account \eqref{interm9} and the positivity of $f$, we readily observe that, on the one hand,
\begin{equation}\label{interm16}
\begin{split}
|Y(\eta)|^{(1-p)/(2-p)}&=r^{2(1-p)/(2-p)}|f'(r)|^{1-p}=r^{2(1-p)/(2-p)}\frac{|X(\eta)|}{rf(r)}\\
&=r^{-p/(2-p)}\frac{|X(\eta)|}{f(r)},
\end{split}
\end{equation}
so that, by~\eqref{p04},
\begin{align*}
	r^{p/(2-p)}f(r)|Y(\eta)|^{-(2p-3)/(2-p)} & =r^{p/(2-p)}f(r)|Y(\eta)|^{(1-p)/(2-p)}|Y(\eta)| \\
	& =|X(\eta)||Y(\eta)|= X(\eta) Y(\eta),
\end{align*}
while, on the other hand,
\begin{align*}
	r^{(p-2q)/(2-p)}|Y(\eta)|^{(q-2p+3)/(2-p)} & = |Y(\eta)|r^{(p-2q)/(2-p)}|Y(\eta)|^{(q-p+1)/(2-p)} \\
	& =|Y(\eta)||Z(\eta)|=Y(\eta) Z(\eta).
\end{align*}
Putting these last calculations together, we arrive to the following autonomous equation for $Y$:
\begin{equation}\label{syst.Y}
\begin{split}
\frac{p-1}{2-p}\dot{Y}(\eta) & = \left[\frac{2(p-1)}{2-p}-N+1\right] Y(\eta) + \alpha X(\eta)Y(\eta) \\
& \quad -\beta|Y(\eta)|Y(\eta)- Y(\eta)Z(\eta).
\end{split}
\end{equation}
Following similar ideas, we deduce from the definition of $X\circ\eta$ in \eqref{change.var} that
\begin{equation}\label{p03}
	f'(r)=-(rf(r))^{1/(p-1)}|(X\circ\eta)(r)|^{-p/(p-1)}(X\circ\eta)(r),
\end{equation}
hence
\begin{equation}\label{interm15bis}
|f'(r)|^{p-2}f'(r)=-rf(r)|(X\circ\eta)(r)|^{-2}(X\circ\eta)(r).
\end{equation}
By taking derivatives with respect to $r$ and taking once more into account that $\eta'(r)=1/r$, we obtain
\begin{equation}\label{interm15}
\begin{split}
(|f'|^{p-2}f')'(r)&=-f(r)|(X\circ\eta)(r)|^{-2}(X\circ\eta)(r)\\
&\quad +r(rf(r))^{1/(p-1)}|(X\circ\eta)(r)|^{-p/(p-1)}\\
&\quad +f(r)|(X\circ\eta)(r)|^{-2}(\dot{X}\circ\eta)(r).
\end{split}
\end{equation}
We then replace the terms involving $f'$ in~\eqref{SSODE} by the expressions given in~\eqref{p03}, \eqref{interm15bis} and~\eqref{interm15}. Setting again $X(\eta):=X\circ\eta$, we obtain after straightforward calculations that
\begin{equation*}
\begin{split}
\dot{X}(\eta)&=NX(\eta)-\alpha|X(\eta)|^2-r^{p/(p-1)}f(r)^{(2-p)/(p-1)}|X(\eta)|^{(p-2)/(p-1)}\\
&\quad +\beta r^{p/(p-1)}f(r)^{(2-p)/(p-1)}|X(\eta)|^{(p-2)/(p-1)}X(\eta)\\
&\quad +r^{q/(p-1)}f(r)^{(q-p+1)/(p-1)}|X(\eta)|^{(2p-2-q)/(p-1)}.
\end{split}
\end{equation*}
We notice again from~\eqref{interm16} that, on the one hand,
$$
r^{p/(p-1)}f(r)^{(2-p)/(p-1)}|X(\eta)|^{(p-2)/(p-1)}=|Y(\eta)|.
$$
On the other hand, using once more~\eqref{interm16}, we can write
\begin{align*}
& r^{q/(p-1)}f(r)^{(q-p+1)/(p-1)}|X(\eta)|^{(p-1-q)/(p-1)} \\
& \hspace{1cm} =r^{q/(p-1)}\left[r^{-p/(2-p)}|Y(\eta)|^{(p-1)/(2-p)}\right]^{(q-p+1)/(p-1)}\\
& \hspace{1cm} =r^{q/(p-1)+p(p-q-1)/(p-1)(2-p)}|Y(\eta)|^{(q-p+1)/(2-p)}\\
& \hspace{1cm}  =r^{(p-2q)/(2-p)}|Y(\eta)|^{(q-p+1)/(2-p)}=|Z(\eta)|.
\end{align*}
Using again~\eqref{p04}, we thus obtain an autonomous equation for $X$, given by
\begin{equation}\label{syst.X}
\dot{X}(\eta)=NX(\eta)-|Y(\eta)|-\alpha X(\eta)^2+\beta X(\eta)|Y(\eta)|+X(\eta)Z(\eta).
\end{equation}
We are left with deriving a differential equation for $Z$. By~\eqref{change.var} and~\eqref{interm11}, we have
\begin{equation*}
\begin{split}
\dot{Z}(\eta)&=\frac{p-2q}{2-p}Z(\eta)+\frac{q-p+1}{2-p}r^{(p-2q)/(2-p)}|Y(\eta)|^{(q-1)/(2-p)}\dot{Y}(\eta)\\
&=\frac{p-2q}{2-p}Z(\eta)+\frac{q-p+1}{p-1}r^{(p-2q)/(2-p)}|Y(\eta)|^{(q-1)/(2-p)}\\
&\quad \times\left[\left(\frac{2(p-1)}{2-p}-(N-1)\right)Y(\eta)-\beta|Y(\eta)|Y(\eta)+\alpha X(\eta)Y(\eta)-Y(\eta)Z(\eta)\right],
\end{split}
\end{equation*}
and taking into account the definition of $Z(\eta)$, we finally obtain
\begin{equation}\label{syst.Z}
\begin{split}
\dot{Z}(\eta)&=\frac{q-p+1}{p-1}Z(\eta)\left[Z_*-Z(\eta)\right]\\
&\quad +\frac{q-p+1}{p-1}\left[\alpha X(\eta)-\beta|Y(\eta)|\right] Z(\eta),
\end{split}
\end{equation}
where $Z_*$ is defined in~\eqref{zstar}. Gathering equations~\eqref{syst.Y}, \eqref{syst.X} and~\eqref{syst.Z} and using~\eqref{p04}, we obtain the autonomous system~\eqref{syst}, as claimed.
\end{proof}

Recalling that $a\in\mathcal{B}$ and taking into account the behavior of $f(r)$ and $f'(r)$ as $r\to\infty$ given in Corollary~\ref{cor.B} and~\eqref{deriv.f}, we infer from~\eqref{exp.eq1} and~\eqref{change.var} that
$$
(X\circ\eta)(r)\sim(\mu K^*)^{1-p}K^*r^{1+(\mu+1)(p-1)-\mu}=(\mu K^*)^{1-p}K^*r^{-(p-2q)/(q-p+1)}\to 0
$$
as $r\to\infty$, since $(p-2q)/(q-p+1)>0$. In a similar manner,
$$
(Y\circ\eta)(r)\sim(\mu K^*)^{2-p}r^{2-(2-p)(\mu+1)}=(\mu K^*)^{2-p}r^{-(p-2q)/(q-p+1)}\to 0
$$
as $r\to\infty$, while
$$
(Z\circ\eta)(r)\sim(\mu K^*)^{q-p+1}=\frac{p-1}{q-p+1}-N+1=Z_* \qquad {\rm as} \ r\to\infty.
$$
Thus, any profile $f(\cdot;a)$ with $a\in\mathcal{B}$ is mapped by the transformation~\eqref{change.var} into a complete orbit $\big(X,Y,Z)(\cdot;a)$ of the system~\eqref{syst} in $[0,\infty)^3$, see~\eqref{p04}, and converging as $\eta\to\infty$ to the critical point
\begin{equation*}
P_0=\left(0,0,Z_*\right).
\end{equation*}
We are now in a position to move to dynamical systems techniques and carefully analyze the stable manifold in a neighborhood of the critical point $P_0$. This is the goal of the next section.

\subsection{Analysis of a stable manifold. Proof of Theorem \ref{th.dimN}}\label{subsec.stable}

We proceed in this section with the analysis of the stable manifold of the critical point $P_0$. Even though we are only interested in orbits coming through the transformation~\eqref{change.var} from profiles $f(\cdot;a)$ with $a\in\mathcal{B}$, we actually study the asymptotic behavior of arbitrary orbits of~\eqref{syst} lying in the stable manifold of $P_0$. More precisely, we start by translating $P_0$ to the origin $\mathbf{0}=(0,0,0)$ of $\real^3$ and introduce $W:=Z-Z_*$ and $\mbx=(X,Y,W)$. Then the system~\eqref{syst} reads
\begin{equation}\label{syst2}
	\dot{\mbx} = \mathbf{F}(\mbx),
\end{equation}
where
\begin{align*}
F_1(\mbx) &:= (N+Z_*)X-Y-\alpha X^2+\beta XY+XW,\\
F_2(\mbx) &:= -\frac{(p-2q)}{q-p+1} Y+\frac{2-p}{p-1}Y(\alpha X-\beta Y-W),\\
F_3(\mbx) &:=\frac{q-p+1}{p-1} Z_* (\alpha X - \beta Y -W) + \frac{q-p+1}{p-1} W(\alpha X - \beta Y -W).
\end{align*}
In order to simplify the notation, we set throughout the remaining part of this section
\begin{equation*}
	\nu:=\frac{q-p+1}{p-1}, \qquad \zeta:=N(p-1)+qZ_*.
\end{equation*}
With this notation, the matrix of the linearization of~\eqref{syst2} near the origin is
\begin{equation*}
	D\mathbf{F}(\mathbf{0})=\begin{pmatrix}
    	N+Z_* & -1 & 0 \\[1mm]
        0 & \displaystyle{-\frac{p-2q}{q-p+1}} & 0 \\[1mm]
        \alpha\nu Z_* & -\beta\nu Z_* & -\nu Z_* \\
	\end{pmatrix}
\end{equation*}
with eigenvalues
\begin{equation}\label{eigenvalues}
\lambda_1:=N+Z_*>0, \qquad \lambda_2:=-\frac{p-2q}{q-p+1}<0, \qquad \lambda_3:=-\nu Z_*<0,
\end{equation}
and corresponding eigenvectors
\begin{equation}\label{eigenvectors}
\mathbf{V}_1:=(\zeta,0,\alpha(q+1-p)Z_*), \qquad \mathbf{V}_2:=(q-p+1,p-q,0), \qquad \mathbf{V}_3:=(0,0,1).
\end{equation}
We thus find that $\mathbf{0}$ is a saddle point of~\eqref{syst2} with a one-dimensional unstable manifold and a two-dimensional stable manifold $\mathcal{W}_s(\mathbf{0})$, and we focus on the dynamics of the orbits belonging to the latter. To proceed further, we denote the semiflow associated to~\eqref{syst2} by $\mathbf{\Phi}$; that is, given $\mbx_0\in \real^3$, $\mathbf{\Phi}(\cdot;\mbx_0)$ is the unique solution to
\begin{equation*}
	\begin{split}
		\dot{\mathbf{\Phi}}(\eta;\mbx_0) & = F(\mathbf{\Phi}(\eta;\mbx_0)), \qquad \eta\in \mathcal{J}_{\mbx_0}, \\
		\mathbf{\Phi}(0;\mbx_0) & = \mbx_0,
	\end{split}
\end{equation*}
defined on a maximal open interval $\mathcal{J}_{\mbx_0}\subset\real$ with $0\in \mathcal{J}_{\mbx_0}$. Since $\mathbf{F}$ is a quadratic polynomial, we point out that $\Phi(\cdot;\mbx_0)\in C^\infty\left( \mathcal{J}_{\mbx_0};\real^3 \right)$.

Since $\mathbf{0}$ is an hyperbolic point for~\eqref{syst2}, it follows from the proof of the stable manifold theorem (see for example \cite[Theorem~19.11]{Amann}) that there exist an open neighborhood $\mathcal{V}\subset\real^3$ of $\mathbf{0}$, two open neighborhoods $\mathcal{U}_0\subset \mathcal{U}\subset\real^2$ of $(0,0)$, and $h\in C^\infty(\mathcal{U},\real)$ such that $h(0,0)=Dh(0,0)=0$ and the local stable manifold
\begin{subequations}\label{p07}
\begin{equation}
\mathcal{W}_s^{\mathcal{V}}(\mathbf{0}) := \left\{ \mbx_0 \in \mathcal{W}_s(\mathbf{0})\ :\ \mathbf{\Phi}(\eta;\mbx_0) \in \mathcal{V} \;\text{ for all }\; \eta\ge 0 \right\} \label{p07a}
\end{equation}
satisfies
\begin{equation}
	\mathcal{W}_s^{\mathcal{V}}(\mathbf{0})  \subset \left\{ u \mathbf{V}_2 + v \mathbf{V}_3 + h(u,v) \mathbf{V}_1\ :\ (u,v)\in\mathcal{U} \right\} \label{p07b}
\end{equation}
and
\begin{equation}
	\left\{ u \mathbf{V}_2 + v \mathbf{V}_3 + h(u,v) \mathbf{V}_1\ :\ (u,v)\in\mathcal{U}_0 \right\} \subset \mathcal{W}_s^{\mathcal{V}}(\mathbf{0}). \label{p07c}
\end{equation}
\end{subequations}
Consider now $\mbx_0\in \mathcal{W}_s^{\mathcal{V}}(\mathbf{0})$. Then $[0,\infty)\subset \mathcal{J}(\mbx_0)$, $\mathbf{\Phi}(\eta;\mbx_0)\in \mathcal{W}_s^{\mathcal{V}}(\mathbf{0})$ for all $\eta\ge 0$ and we infer from this invariance property and~\eqref{p07} that there are functions $(U,V):[0,\infty)\to\mathcal{U}$ such that $(X,Y,W):= \mathbf{\Phi}(\cdot;\mbx_0)$ satisfies
\begin{equation}\label{decomp}
(X,Y,W)(\eta)  = U(\eta) \mathbf{V}_2 + V(\eta) \mathbf{V}_3 + h(U(\eta),V(\eta)) \mathbf{V}_1, \qquad \eta\ge 0.
\end{equation}
We readily deduce from~\eqref{decomp} that
\begin{subequations}\label{stable.man}
\begin{equation}\label{eq.X}
	X=(q-p+1)U+\zeta h(U,V),
\end{equation}
\begin{equation}\label{eq.Y}
	Y=(p-q)U,
\end{equation}
\begin{equation}\label{eq.W}
	W=V+\alpha(q-p+1)Z_*h(U,V),
\end{equation}
whence
\begin{equation}
	V = - \frac{\alpha (q-p+1) Z_*}{\zeta} X + \frac{\alpha (q-p+1)^2 Z_*}{(p-q)\zeta} Y + W. \label{eq.S}
\end{equation}
\end{subequations}
Observe that~\eqref{eq.Y} and~\eqref{eq.S} ensure that $U$ and $V$ both belong to $C^\infty([0,\infty))$ which, in turn, implies that $h(U,V)\in C^\infty([0,\infty))$ as a consequence of the regularity of $h$. To determine the behavior of $(X,Y,W)$, we set
\begin{equation}
	H  :=h(U,V), \qquad	G :=\alpha X - \beta Y - W = -V + \frac{\alpha q}{\nu} H. \label{p08}
\end{equation}

\begin{lemma}\label{lem.stable}
Let $\mbx_0\in\mathcal{W}_s^{\mathcal{V}}(\mathbf{0})$. Then the functions $(U,V)$ defined in~\eqref{decomp} solve
\begin{equation}\label{eq.U}
\dot{U}=\lambda_2U+\frac{2-p}{p-1}UG,
\end{equation}
and
\begin{equation}\label{eq.V}
\dot{V}=\lambda_3V+\frac{\alpha\nu(q-p+1)Z_*}{\zeta}UG+\nu VG+\alpha\nu qZ_*GH,
\end{equation}
in $(0,\infty)$, where $\lambda_2$ and $\lambda_3$ are the eigenvalues defined in~\eqref{eigenvalues}.
\end{lemma}

\begin{proof}
Using the invariance of the stable manifold, we deduce from the second equation in~\eqref{syst2} and~\eqref{eq.Y} that, on the one hand,
$$
(p-q)\dot{U}=\lambda_2(p-q)U+\frac{(2-p)(p-q)}{p-1}UG,
$$
which readily gives~\eqref{eq.U}. On the other hand, the first equation in~\eqref{syst2} and~\eqref{eq.X} give
\begin{align*}
(q-p+1)\dot{U} + \zeta\dot{H} & = (N+Z_*) \big[(q-p+1)U + \zeta H\big] \\
& \qquad -(p-q)U- \big[(q-p+1)U+\zeta H\big]G.
\end{align*}
Taking into account that
\begin{equation*}
	(N+Z_*)(q-p+1)-(p-q)=-(p-2q),
\end{equation*}
we further obtain
\begin{equation*}
	\dot{U}+\frac{\zeta}{q-p+1}\dot{H}=\lambda_2U+\frac{(N+Z_*)\zeta}{q-p+1}H-\left[U+\frac{\zeta}{q-p+1}H\right]G.
\end{equation*}
We next replace $\dot{U}$ in the previous equation with its expression in the already established~\eqref{eq.U} to get
\begin{equation}\label{eq.H}
\dot{H}=-\frac{\nu}{\zeta}UG+(N+Z_*)H-GH.
\end{equation}
Finally, we employ the third equation in~\eqref{syst2} and~\eqref{eq.W} in order to obtain the equation for $V$. More precisely, we have
\begin{equation*}
\begin{split}
\dot{V}+\alpha(q-p+1)Z_*\dot{H}&=\nu Z_*G+\nu[V+\alpha(q-p+1)Z_*H]G\\
&=\nu Z_*\left[-V+\frac{\alpha q}{\nu}H\right] +\nu [V+\alpha(q-p+1)Z_*H]G\\
&=\lambda_3V+\alpha qZ_*H+\nu [V+\alpha(q-p+1)Z_*H]G.
\end{split}
\end{equation*}
We replace now $\dot{H}$ by the identity~\eqref{eq.H} to obtain
\begin{equation*}
\begin{split}
\dot{V}&=\lambda_3V-\alpha(q-p+1)Z_*\left[-\frac{\nu}{\zeta}UG+(N+Z_*)H-GH\right]\\
&\quad +\alpha qZ_*H+\nu [V+\alpha(q-p+1)Z_*H] G,
\end{split}
\end{equation*}
from which~\eqref{eq.V} follows after straightforward manipulations and taking into account the identity
$\alpha qZ_*=\alpha(q-p+1)(N+Z_*)Z_*$.
\end{proof}

One further difficulty in the forthcoming analysis is the fact that the eigenvalues $\lambda_2$ and $\lambda_3$ are not always ordered in the same way. Indeed, a direct calculation leads to
$$
\lambda_2-\lambda_3=-\frac{P(q-p+1)}{(p-1)(q-p+1)}, \qquad P(\lambda):=(N-1)\lambda^2-3(p-1)\lambda+(2-p)(p-1),
$$
and we may observe that
$$
P(0)=(2-p)(p-1)>0, \qquad P\left(\frac{2-p}{2}\right)=\frac{(2-p)(N+1)}{4}(p_c-p)<0,
$$
and $P'<0$ on $[0,(2-p)/2]$. Hence there is a unique $\lambda_*\in(0,(2-p)/2)$ satisfying $P(\lambda_*)=0$. We thus deduce that there exists a unique $q_*:=\lambda_*+p-1\in(p-1,p/2)$ such that $\lambda_2=\lambda_3$ for $q=q_*$. This interchange of order between the two negative eigenvalues of $D\mathbf{F}(\mathbf{0})$ generates some additional technical difficulties in the (very careful) analysis we perform below. We continue our analysis of the trajectories on the stable manifold by the following boundedness result.

\begin{lemma}\label{lem.bound}
Let $\mbx_0\in\mathcal{W}_s^{\mathcal{V}}(\mathbf{0})$. There exists a constant $C_0>0$ depending on $N$, $p$, $q$, and $\mbx_0$ such that the functions $(U,V)$ defined in~\eqref{decomp} satisfy
\begin{equation*}
	|U(\eta)|+|V(\eta)|\leq C_0 e^{\Lambda\eta}, \qquad \Lambda:=\max\{\lambda_2,\lambda_3\}<0,
\end{equation*}
for any $\eta\in[0,\infty)$.
\end{lemma}

\begin{proof}
We multiply~\eqref{eq.U} by $\mathrm{sign}(U)$ and~\eqref{eq.V} by $\mathrm{sign}(V)$ and add the resulting identities to obtain
\begin{align*}
	\frac{d}{d\eta}(|U|+|V|)&= \lambda_2 |U| + \frac{2-p}{p-1} |U| G +\lambda_3 |V| + \frac{\alpha\nu(q-p+1)Z_*}{\zeta}UG\, \mathrm{sign}(V) \\
	& \quad +\nu |V|G + \alpha\nu qZ_* GH\, \mathrm{sign}(V)\\
	&\leq \Lambda(|U|+|V|)+\frac{2-p}{p-1}|U||G|+\frac{\alpha\nu(q-p+1)Z_*}{\zeta}|U||G|\\
	&\quad +\nu|V||G|+\alpha\nu qZ_*|G||H|.
\end{align*}
At this point, since $h(0,0)=0$ and $(U,V)\in L^\infty((0,\infty);\real^2)$, we note that there is $C_1>0$ depending on $N$, $p$, $q$, and $\mbx_0$ such that
\begin{equation}\label{bound.H}
|H|=|h(U,V)|\leq C_1(|U|+|V|),
\end{equation}
while
\begin{equation}\label{bound.G}
	|G|\leq |V| + \frac{\alpha q}{\nu}|H|\leq C_2(|U|+|V|), \qquad C_2 := 1 + \frac{\alpha q}{\nu} C_1.
\end{equation}
Consequently, there is $C_3>0$ depending on $N$, $p$, $q$, and $\mbx_0$ such that
\begin{equation}\label{interm27}
\frac{d}{d\eta}(|U|+|V|)+|\Lambda|(|U|+|V|)\leq C_3(|U|+|V|)^2, \qquad \eta\geq0.
\end{equation}
Setting $\Sigma:=|U|+|V|$, we infer from $\mbx_0 \in \mathcal{W}_s^{\mathcal{V}}(\mathbf{0})$ and~\eqref{stable.man} that $\Sigma(\eta)\to 0$ as $\eta\to\infty$, so that there is $\eta_0\ge 0$ such that $\Sigma(\eta)\leq|\Lambda|/(2C_3)$ for $\eta\geq\eta_0$. It follows from~\eqref{interm27} that
$$
\dot{\Sigma}(\eta)\leq-\Sigma(\eta)(|\Lambda|-C_3\Sigma(\eta)), \qquad \eta\geq\eta_0,
$$
and we obtain by integration, taking into account the non-positivity of the right-hand side, that
$$
\frac{\Sigma(\eta)}{|\Lambda|}\le \frac{\Sigma(\eta)}{|\Lambda|-C_3\Sigma(\eta)}\leq C_4e^{-|\Lambda|\eta}, \qquad C_4=\frac{\Sigma(\eta_0)}{|\Lambda|-C_3\Sigma(\eta_0)}e^{|\Lambda|\eta_0},
$$
for $\eta\geq\eta_0$. The conclusion follows from the previous estimate and the boundedness of $\Sigma$ on $[0,\eta_0]$.
\end{proof}

The convergence for $U$ follows now rather easily from~\eqref{eq.U}.

\begin{lemma}\label{lem.U}
Let $\mbx_0\in\mathcal{W}_s^{\mathcal{V}}(\mathbf{0})$. There exists $U_{\infty}(\mbx_0)\in\real$ such that  the function $U$ defined in~\eqref{decomp} satisfies
$$
\lim\limits_{\eta\to\infty}e^{-\lambda_2\eta}U(\eta)=U_{\infty}(\mbx_0).
$$
\end{lemma}

\begin{proof}
We readily infer from~\eqref{eq.U} that
$$
\frac{d}{d\eta}\left[U(\eta)\exp\left(-\lambda_2\eta-\frac{2-p}{p-1}\int_{0}^{\eta}G(s)\,ds\right)\right]=0,
$$
whence
\begin{equation*}
e^{-\lambda_2\eta}U(\eta)= U(0)  \exp\left(\frac{2-p}{p-1}\int_{0}^{\eta}G(s)\,ds\right), \qquad \eta\geq 0.
\end{equation*}
We then infer from Lemma~\ref{lem.bound} and the estimate~\eqref{bound.G} that $G\in L^1(0,\infty)$, so that
$$
\lim\limits_{\eta\to\infty}e^{-\lambda_2\eta}U(\eta)= U(0) \exp\left(\frac{2-p}{p-1}\int_{0}^{\infty}G(s)\,ds\right)=:U_{\infty}(\mbx_0),
$$
and the proof is complete.
\end{proof}

Obtaining a similar result for the component $V$ of the orbit, with respect to the eigenvalue $\lambda_3$, is much more involved, as~\eqref{eq.V} involves more terms. Thus, we need one more technical, preparatory lemma about the Taylor expansion of the function $h$ at $(0,0)$.

\begin{lemma}\label{lem.Taylor}
For any integer $n\geq1$, we have
$$
\partial_u^n h(0,0)=\partial_v^n h(0,0)=0.
$$
\end{lemma}

\begin{proof}
The proof is based on a rather technical computation of the coefficients of the Taylor expansion of the function $h$. To this end, we pick an arbitrary pair $(u,v)\in \mathcal{U}_0$ and recall that~\eqref{p07c} ensures that
\begin{equation*}
	\mbx_0=u\mathbf{V}_2 +  v\mathbf{V}_3 + h(u,v)\mathbf{V}_1 \in \mathcal{W}_s^{\mathcal{V}}(\mathbf{0}).
\end{equation*}
Keeping the notation $(X,Y,W) = \mathbf{\Phi}(\cdot;\mbx_0)$ and the corresponding $(U,V,H,G)$ introduced in~\eqref{decomp} and~\eqref{p08}, we deduce from~\eqref{p08} and~\eqref{eq.H} that
\begin{align*}
\partial_u h(U,V) \dot{U} + \partial_v h(U,V) \dot{V} & = \frac{\nu}{\zeta} \left[UV-\frac{\alpha(p-1)q}{q-p+1}UH\right]\\
& \qquad +(N+Z_*)H+\left[V-\frac{\alpha(p-1)q}{q-p+1}H\right]H.
\end{align*}
We next replace $\dot{U}$ and $\dot{V}$ from~\eqref{eq.U} and~\eqref{eq.V} into the previous identity. Taking also into account~\eqref{p08}, the expression of $Z_*$ in \eqref{zstar}, and letting $\eta\to0$ we are thus left, after some easy manipulations, with the following (lengthy) equality which holds true for any $(u,v)\in\mathcal{U}_0$:
\begin{equation}\label{interm29}
\begin{split}
&\partial_uh(u,v)\left[\lambda_2u-\frac{2-p}{p-1}uv+\frac{\alpha q(2-p)}{q-p+1}uh(u,v)\right]\\
&+\partial_vh(u,v)\left[\lambda_3v-\frac{\alpha\nu(q-p+1)Z_*}{\zeta}uv-\nu v^2\right.\\
&\hspace{1cm} \left. + \frac{\alpha^2q(q-p+1)}{\zeta} uh(u,v)+\alpha\nu(N-1)qvh(u,v)+\alpha^2q^2Z_*h^2(u,v)\right]\\
&=\frac{\nu}{\zeta}uv+\frac{q}{q-p+1}h(u,v)-\frac{\alpha q}{\zeta}uh(u,v)+vh(u,v)-\frac{\alpha q(p-1)}{q-p+1}h^2(u,v).
\end{split}
\end{equation}
We let first $v=0$ in~\eqref{interm29} and, taking into account that $h(0,0)=Dh(0,0)=0$ and thus that $h(u,v)$ is at least a quadratic expression in $(u,v)$, we obtain by identifying only the quadratic terms:
\begin{equation}\label{interm30}
\lambda_2u\partial_uh(u,0)=\frac{q}{q-p+1}h(u,0)+o(h(u,0)).
\end{equation}
Since, for any $n\geq1$,
\begin{equation*}
\begin{split}
&h(u,0)=\sum\limits_{k=0}^n\frac{\partial^k_uh(0,0)}{k!}u^k+o(u^n),\\
&\partial_uh(u,0)=\sum\limits_{k=0}^{n-1}\frac{\partial^{k+1}_uh(0,0)}{k!}u^k+o(u^{n-1}),
\end{split}
\end{equation*}
we infer from~\eqref{interm30} that
$$
\lambda_2\sum\limits_{k=1}^n\frac{\partial^k_uh(0,0)}{(k-1)!}u^k-\frac{q}{q-p+1}\sum\limits_{k=0}^n\frac{\partial^k_uh(0,0)}{k!}u^k=o(h(u,0))+o(u^n),
$$
or equivalently,
\begin{equation}\label{interm31}
\sum\limits_{k=2}^n\left[\frac{\lambda_2}{(k-1)!}-\frac{q}{(q-p+1)k!}\right]\partial^k_uh(0,0)u^k=o(u^n)+o(h(u,0)).
\end{equation}
We proceed by induction. For $n=2$, we obtain from~\eqref{interm31} that
$$
\left[\lambda_2-\frac{q}{2(q-p+1)}\right]\partial^2_uh(0,0)u^2=o(u^2)+o(h(u,0))=o(u^2),
$$
and we infer from the negativity of $\lambda_2$ that $\partial^2_uh(0,0)=0$. Assume next that $\partial^k_uh(0,0)=0$ for $0\leq k\leq n-1$, $n\geq3$. Then $h(u,0)=O(u^n)$, so that $o(h(u,0))=o(u^n)$ and it follows from~\eqref{interm31} that
$$
\left[\frac{\lambda_2}{(n-1)!}-\frac{q}{(q-p+1)n!}\right]\partial^n_uh(0,0)u^n=o(u^n).
$$
Consequently, using again the negativity of $\lambda_2$, we get that $\partial^n_uh(0,0)=0$ and the proof by induction is completed. We proceed in a similar way for the derivatives with respect to the $v$ variable, by letting $u=0$ in~\eqref{interm29} and identifying only the quadratic terms to find
\begin{equation}\label{interm32}
\lambda_3 v\partial_vh(0,v)=\frac{q}{q-p+1}h(0,v)+o(h(0,v)).
\end{equation}
We plug again the Taylor expansion of $h(0,v)$ into~\eqref{interm32}, as we did before with $h(u,0)$. Indeed, since
\begin{equation*}
\begin{split}
&h(0,v)=\sum\limits_{k=0}^n\frac{\partial^k_vh(0,0)}{k!}v^k+o(v^n),\\
&\partial_vh(0,v)=\sum\limits_{k=0}^{n-1}\frac{\partial^{k+1}_vh(0,0)}{k!}v^k+o(v^{n-1}),
\end{split}
\end{equation*}
we derive from~\eqref{interm32} that
\begin{equation}\label{interm33}
\sum\limits_{k=2}^n\left[\frac{\lambda_3}{(k-1)!} - \frac{q}{(q-p+1)k!} \right] \partial^k_vh(0,0) = o(v^n)+o(h(0,v)).
\end{equation}
We proceed again by induction. For $n=2$, \eqref{interm33} gives
\begin{equation*}
	\lambda_3\partial^2_vh(0,0)v^2=\frac{q}{2(q-p+1)}\partial^2_vh(0,0)v^2+o(v^2)
\end{equation*}
and the negativity of $\lambda_3$ implies that $\partial^2_vh(0,0)=0$. Assuming as before that $\partial^k_vh(0,0)=0$ for $0\leq k\leq n-1$ and $n\geq 3$ and replacing this induction assumption into~\eqref{interm33}, we are left with
\begin{equation*}
\left[ \frac{\lambda_3}{(n-1)!} -\frac{q}{(q-p+1)n!} \right] \partial^n_vh(0,0)v^n=o(v^n),
\end{equation*}
and once more the negativity of $\lambda_3$ implies that $\partial^n_vh(0,0)=0$, completing the induction step and the proof.
\end{proof}

With the previous preparation, we can now state and prove a convergence lemma for $V$.

\begin{lemma}\label{lem.V}
Let $\mbx_0\in\mathcal{W}_s^{\mathcal{V}}(\mathbf{0})$. There exists $V_{\infty}(\mbx_0)\in\real$ such that the function $V$ defined in~\eqref{decomp} satisfies
$$
\lim\limits_{\eta\to\infty}e^{-\lambda_3\eta}V(\eta)=V_{\infty}(\mbx_0).
$$
\end{lemma}

\begin{proof}
Let $n\ge 2$ be such that $n\Lambda < \lambda_3$. Owing to Lemma~\ref{lem.Taylor} and Taylor's theorem, there exists $C_5(n)>0$ depending only on $N$, $p$, $q$, $\mbx_0$, and $n$ such that
\begin{equation}\label{interm34}
	|H| = |h(U,V)|\leq C_5(n) (|U|+|V|)^n.
\end{equation}
Recalling~\eqref{p08}, we next notice that we can write~\eqref{eq.V} in the form
\begin{equation*}
	\dot{V} = \big( \lambda_3+Q_1 \big) V + Q_2,
\end{equation*}
where
\begin{align*}
	Q_1 & := -\frac{\alpha\nu(q-p+1)Z_*}{\zeta}U - \nu V, \\
	Q_2 & := \frac{\alpha^2(q-p+1)qZ_*}{\zeta} UH + \alpha q\nu (N-1) VH + \alpha^2 q^2 Z_* H^2.
\end{align*}
Consequently, for $\eta\ge 0$,
\begin{align*}
	\frac{d}{d\eta} \left[ V(\eta) \exp{\left( - \lambda_3\eta - \int_0^\eta Q_1(s)\,ds \right)} \right] = Q_2(\eta) \exp{\left( - \lambda_3\eta - \int_0^\eta Q_1(s)\,ds \right)},
\end{align*}
whence, after integration,
\begin{equation}
	\begin{split}
		e^{-\lambda_3\eta} V(\eta) & = V(0) \exp{\left( \int_0^\eta Q_1(s)\,ds \right)} \\
		& \quad + \int_0^\eta e^{-\lambda_3 s} Q_2(s) \exp{\left( \int_s^\eta Q_1(s_*)\,ds_* \right)}\, ds.
	\end{split} \label{p09}
\end{equation}
On the one hand, we deduce from Lemma~\ref{lem.bound} that, for $\eta\ge 0$,
\begin{equation*}
	|Q_1(\eta)| \le \nu \left( 1 + \frac{\alpha (q-p+1)Z_*}{\zeta} \right) C_0 e^{\Lambda\eta},
\end{equation*}
so that $Q_1\in L^1(0,\infty)$. On the other hand, setting
\begin{equation*}
	C_6 := \alpha q \max\left\{ \frac{\alpha (q-p+1) Z_*}{\zeta} , \nu (N-1) , \alpha q Z_*\right\},
\end{equation*}
we infer from~\eqref{bound.H}, \eqref{interm34} and Lemma~\ref{lem.bound} that, for $\eta\ge 0$,
\begin{align*}
	e^{-\lambda_3 \eta} |Q_2(\eta)| & \le C_6 \big( |U(\eta)|+|V(\eta)| +  |H(\eta)| \big) |H(\eta)| e^{-\lambda_3 \eta} \\
	& \le C_6 (1+C_1) C_5(n) \big( |U(\eta)|+|V(\eta)| \big)^{n+1} e^{-\lambda_3 \eta} \\
	& \le C_6 (1+C_1) C_5(n) C_0^{n+1} e^{[(n+1)\Lambda -\lambda_3] \eta} \\
	& \le C_6 (1+C_1) C_5(n) C_0^{n+1} e^{\Lambda \eta},
\end{align*}
and this upper bound guarantees that $\eta\mapsto e^{-\lambda_3 \eta} Q_2(\eta)$ belongs to $L^1(0,\infty)$. Consequently,
\begin{equation*}
	V_\infty(\mbx_0) := V(0) \exp{\left( \int_0^\infty Q_1(s)\,ds \right)} \quad + \int_0^\infty e^{-\lambda_3 s} Q_2(s) \exp{\left( \int_s^\infty Q_1(s_*)\,ds_* \right)}\, ds
\end{equation*}
is finite and we let $\eta\to\infty$ in~\eqref{p09} to complete the proof of Lemma~\ref{lem.V}.
\end{proof}

Putting together the previous analysis, we can now identify the decay as $\eta\to\infty$ on any trajectory contained in the stable manifold $\mathcal{W}_s^{\mathcal{V}}(\mathbf{0})$ of the critical point $P_0$ of the system~\eqref{syst2}.

\begin{proposition}\label{prop.stable}
Let $\mbx_0\in \mathcal{W}_s^{\mathcal{V}}(\mathbf{0})$. Then, as $\eta\to\infty$,
\begin{equation}\label{beh.stable}
\mbx(\eta;\mbx_0) = \left[ U_{\infty}(\mbx_0) e^{\lambda_2\eta} + o\left(e^{\lambda_2\eta}\right) \right] \mathbf{V}_2 + \left[ V_{\infty}(\mbx_0) e^{\lambda_3\eta} + o\left(e^{\lambda_3\eta}\right) \right] \mathbf{V}_3,
\end{equation}
where $U_{\infty}(\mbx_0)$ and $V_{\infty}(\mbx_0)$ are defined in Lemma~\ref{lem.U} and Lemma~\ref{lem.V}, respectively.
\end{proposition}

\begin{proof}
We first observe that~\eqref{interm34}, along with Lemma~\ref{lem.bound}, implies that, for all $n\ge 2$,
\begin{equation}
	|h(U(\eta),V(\eta))| \le C_0^n C_5(n) e^{n\Lambda\eta}, \qquad \eta\ge 0. \label{p10}
\end{equation}
The asymptotic expansion~\eqref{beh.stable} as $\eta\to\infty$ now follows immediately from~\eqref{stable.man} and~\eqref{p10}, together with the convergences in Lemmas~\ref{lem.U} and~\ref{lem.V}.
\end{proof}

The information given in Proposition~\ref{prop.stable} is already sufficient in order to complete the proof of Theorem~\ref{th.dimN}. But we postpone for the moment its proof and first show that the limits in Lemmas~\ref{lem.U} and~\ref{lem.V} determine uniquely the trajectory contained in the stable manifold of $P_0$.

\begin{proposition}\label{prop.uniq}
The map $\mbx_0 \mapsto \big( U_\infty(\mbx_0),V_\infty(\mbx_0) \big)$ is one-to-one on $\mathcal{V}$, where $U_{\infty}$ and $V_{\infty}$ are introduced in Lemmas~\ref{lem.U} and~\ref{lem.V}, respectively. Moreover, if $\varrho\in\real$ is such that $(\varrho\beta,\varrho\alpha,0)\in \mathcal{V}$, then
\begin{equation*}
	U_\infty((\varrho\beta,\varrho\alpha,0)) = \frac{\varrho}{p-2q}, \qquad V_\infty((\varrho\beta,\varrho\alpha,0)) = 0.	
\end{equation*}
\end{proposition}

\begin{proof}
By Hartman's theorem \cite{Ha60} (see also \cite[Section~2.8]{Pe}), there exist a neighborhood $\mathcal{N}\subset\real^3$ containing the origin and a $C^1$-diffeomorphism $\Psi:\mathcal{N}\mapsto\Psi(\mathcal{N})$ such that $\Psi(\mathbf{0})=\mathbf{0}$ and, for any $\mbx_0\in\mathcal{N}$, there is an open interval $I_{\mbx_0}\subset\real$ containing zero such that
\begin{equation}\label{conjugacy}
\Psi\circ\mathbf{\Phi}(\eta;\mbx_0)=e^{D\textbf{F}(\mathbf{0})\eta}\Psi(\mbx_0),\qquad \eta\in I_{\mbx_0},
\end{equation}
recalling that $\mathbf{\Phi}$ is the flow associated to the dynamical system~\eqref{syst2}. In particular, if we decompose $\Psi$ on the basis formed by the eigenvectors $\{\mathbf{V}_1,\mathbf{V}_2,\mathbf{V}_3\}$ defined in~\eqref{eigenvectors}, then
\begin{equation*}
	\Psi=\sum\limits_{i=1}^3\Psi_i \mathbf{V}_i,
\end{equation*}
and the conjugacy~\eqref{conjugacy} gives
\begin{equation}\label{conj2}
	\Psi_i(\Phi(\eta;\mbx_0))=\Psi_i(\mbx_0) e^{\lambda_i\eta}, \qquad i\in\{1,2,3\}, \qquad \eta\in I_{\mbx_0}.
\end{equation}

Consider now $\mbx_{0,k}\in \mathcal{W}_s^{\mathcal{V}}(\mathbf{0})$, $k\in\{1,2\}$, such that
\begin{equation}
	\big( U_\infty(\mbx_{0,1}) , V_\infty(\mbx_{0,1}) \big) = \big( U_\infty(\mbx_{0,2}) , V_\infty(\mbx_{0,2}) \big) =:(U_\infty,V_\infty). \label{p15}
\end{equation}
Owing to the definition of $\mathcal{W}_s^{\mathcal{V}}(\mathbf{0})$, there is $\eta_0\ge 0$ such that $\mathbf{\Phi}(\eta;\mbx_{0,k})\in\mathcal{N}$ for all $\eta\ge \eta_0$ and $k\in\{1,2\}$. Therefore, setting $\mbx_{k}:= \mathbf{\Phi}(\eta_0;\mbx_{0,k})$ for $k\in\{1,2\}$,  we have $[0,\infty) \subset I_{\mbx_{k}}$ and we infer from~\eqref{conj2} that
\begin{equation}
	\Psi_i(\Phi(\eta;\mbx_{k}))=\Psi_i(\mbx_{k}) e^{\lambda_i \eta}, \qquad i\in\{1,2,3\}, \qquad \eta\geq 0. \label{p11}
\end{equation}
Moreover, since $\Psi\in C^1(\mathcal{N})$ with $\Psi(\mathbf{0})=\mathbf{0}$, we deduce from Proposition~\ref{prop.stable} the following expansion as $\eta\to\infty$
\begin{equation}\label{exp.flow}
\begin{split}
\Psi_i(\Phi(\eta;\mbx_{k}))&\sim\partial_X\Psi_i(0)X_k(\eta)+\partial_Y\Psi_i(0)Y_k(\eta)+\partial_W\Psi_i(0)W_k(\eta)\\
&\sim\partial_X\Psi_i(0)(q-p+1)U_\infty(\mbx_{k}) e^{\lambda_2\eta} + \partial_Y\Psi_i(0)(p-q)U_\infty(\mbx_{k}) e^{\lambda_2\eta} \\
& \quad +\partial_W\Psi_i(0) V_\infty(\mbx_{k}) e^{\lambda_3\eta} \\
&\sim \omega_{i,2} U_\infty(\mbx_{k}) e^{\lambda_2\eta} + \omega_{i,3} V_\infty(\mbx_{k}) e^{\lambda_3\eta},
\end{split}
\end{equation}
for each $i\in\{1,2,3\}$ and $k\in\{1,2\}$, where $(X_k,Y_k,W_k)=\mathbf{\Phi}(\cdot;\mbx_{k})$ and
\begin{equation*}
	\omega_{i,j} := \langle\nabla\Psi_i(0),\mathbf{V}_j\rangle, \qquad (i,j)\in \{1,2,3\}\times \{2,3\}.
\end{equation*}
In addition,
\begin{equation}
	U_\infty(\mbx_k) = U_\infty e^{\lambda_2 \eta_0} \;\;\text{ and }\;\; V_\infty(\mbx_k) = V_\infty e^{\lambda_3 \eta_0}, \qquad k\in\{1,2\}, \label{p16}
\end{equation}
by~\eqref{p15}, since $\mathbf{\Phi}(\eta;\mbx_{0,k}) = (X_k,Y_k,W_k)(\eta-\eta_0)$ for $\eta\ge0$ and $k\in\{1,2\}$. Combining~\eqref{p11}, \eqref{exp.flow}, and~\eqref{p16} gives
\begin{equation}
	\Psi_i(\mbx_{k}) e^{\lambda_i \eta} \sim \omega_{i,2} U_\infty e^{\lambda_2(\eta_0+\eta)} + \omega_{i,3} V_\infty e^{\lambda_3(\eta_0+\eta)} \label{p12}
\end{equation}
as $\eta\to\infty$ for each $i\in\{1,2,3\}$ and $k\in\{1,2\}$. Since $\lambda_1>0$, an immediate consequence of~\eqref{p12} applied with $i=1$ is that
\begin{equation*}
		\Psi_1(\mbx_{k}) = 0, \qquad k\in\{1,2\}.
\end{equation*}
It next follows from~\eqref{p12} applied with $i=2$ and $i=3$ that, as $\eta\to\infty$,
\begin{align}
	\Psi_2(\mbx_{k}) & \sim \omega_{2,2} U_\infty e^{\lambda_2 \eta_0} + \omega_{2,3} V_\infty e^{\lambda_3 \eta_0} e^{(\lambda_3-\lambda_2)\eta}, \label{p14a}\\
	\Psi_3(\mbx_{k}) & \sim \omega_{3,2} U_\infty e^{\lambda_2 \eta_0} e^{(\lambda_2-\lambda_3)\eta} + \omega_{3,3} V_\infty e^{\lambda_3 \eta_0}, \label{p14b}
\end{align}
for $k\in\{1,2\}$. At this stage, we split the rest of the analysis into three cases, according to the sign of $\lambda_2-\lambda_3$.

\medskip

\noindent \textbf{Case 1: $\lambda_3<\lambda_2$}. In this case, we readily infer from~\eqref{p14a} that
\begin{equation*}
	\Psi_2(\mbx_{k}) =\omega_{2,2} U_{\infty} e^{\lambda_2 \eta_0}, \qquad k\in\{1,2\},
\end{equation*}
while~\eqref{p14b} entails
\begin{equation*}
	\omega_{3,2} U_\infty = 0, \qquad \Psi_3(\mbx_{k}) = \omega_{3,3} V_\infty e^{\lambda_3 \eta_0}, \qquad k\in\{1,2\}.
\end{equation*}
Consequently
\begin{equation}\label{interm35}
	\mbx_{1} = \mbx_{2} = \Psi^{-1}\left( 0,\omega_{2,2} U_{\infty} e^{\lambda_2 \eta_0},\omega_{3,3} V_{\infty}e^{\lambda_3 \eta_0} \right),
\end{equation}
and the Cauchy-Lipschitz theorem implies that $\mbx_{0,1}=\mbx_{0,2}$, as claimed.

\medskip

\noindent \textbf{Case 2: $\lambda_2=\lambda_3$}. In that case, it readily follows from~\eqref{p14a} and~\eqref{p14b} that, for $k\in\{1,2\}$,
\begin{align*}
	\Psi_2(\mbx_{k}) & = \omega_{2,2} U_\infty e^{\lambda_2 \eta_0} + \omega_{2,3} V_\infty e^{\lambda_3 \eta_0} , \\
	\Psi_3(\mbx_{k}) & = \omega_{3,2} U_\infty e^{\lambda_2 \eta_0} + \omega_{3,3} V_\infty e^{\lambda_3 \eta_0},
\end{align*}
so that
\begin{equation*}
	\mbx_{1} = \mbx_{2} = \Psi^{-1}\left( 0, \omega_{2,2} U_\infty e^{\lambda_2 \eta_0} + \omega_{2,3} V_\infty e^{\lambda_3 \eta_0} , \omega_{3,2} U_\infty e^{\lambda_2 \eta_0} + \omega_{3,3} V_\infty e^{\lambda_3 \eta_0} \right).
\end{equation*}
Thus, $\mbx_{0,1}=\mbx_{0,2}$ by the Cauchy-Lipschitz theorem.

\medskip

\noindent \textbf{Case 3: $\lambda_3>\lambda_2$}. It is very similar to Case~1. Indeed, we deduce from~\eqref{p14a} and~\eqref{p14b} that, for $k\in\{1,2\}$,
\begin{equation*}
		\Psi_2(\mbx_{k}) = \omega_{2,2} U_\infty e^{\lambda_2 \eta_0}, \quad \omega_{2,3} V_\infty = 0, \quad \Psi_3(\mbx_{k}) = \omega_{3,3} V_\infty e^{\lambda_3 \eta_0},
\end{equation*}
and we arrive again to~\eqref{interm35}, thereby completing the proof of the first statement in Proposition~\ref{prop.uniq}.

Consider now $\varrho\in\real$. Then direct computations reveal that
\begin{equation*}
	\Phi(\eta;(\varrho\beta,\varrho\alpha,0)) = (\varrho\beta,\varrho\alpha,0)e^{\lambda_2 \eta} \;\;\text{ for }\;\; \eta\in\real,
\end{equation*}
and the claim readily follows from~\eqref{stable.man} and this explicit formula.
\end{proof}

We are now in a position to complete the proof of Theorem~\ref{th.dimN}.

\begin{proof}[Proof of Theorem~\ref{th.dimN}]
By Corollary~\ref{cor.B}, the set $\mathcal{B}$ defined in Section~\ref{sec.exist} is non-empty and there is thus $a\in\mathcal{B}$ such that the solution $f(\cdot;a)$ to~\eqref{CP} satisfies $R(a)=\infty$ and
\begin{equation*}
	\lim\limits_{r\to\infty} r^\mu f(r;a) = K^*.
\end{equation*}
Let us now prove the expansion~\eqref{tail.SSS}. To this end, we recall that, according to the discussion closing Section~\ref{subsec.dynamic}, $\mbx_a := \big( X(\cdot;a),Y(\cdot;a),Z(\cdot;a)-Z_* \big)$ defined in~\eqref{change.var} is a complete orbit of~\eqref{syst2} which is included in $\mathcal{W}_s(\mathbf{0})$. In particular, there is $\eta_a\in\real$ such that $\mbx_a(\eta)\in \mathcal{W}_s^{\mathcal{V}}(\mathbf{0})$ for all $\eta\ge\eta_a$ and we infer from Proposition~\ref{prop.stable} that there are $\big(U_{\infty,a},V_{\infty,a}\big)\in\real^2$ such that
\begin{equation}
	\mbx_a(\eta) = \left[ U_{\infty,a} e^{\lambda_2 \eta} + o(e^{\lambda_2 \eta}) \right] \mathbf{V}_2 + \left[ V_{\infty,a} e^{\lambda_3 \eta} + o(e^{\lambda_3 \eta}) \right] \mathbf{V}_3 \qquad {\rm as} \ \eta\to\infty. \label{p17}
\end{equation}
Undoing the transformation~\eqref{change.var} for $Y(\cdot;a)$, we deduce from~\eqref{p17} that
\begin{equation*}
	Y(\ln{r};a)=r^2(-f'(r;a))^{2-p}\sim(p-q)U_{\infty,a}r^{\lambda_2} \qquad {\rm as} \ r\to\infty,
\end{equation*}
or equivalently
\begin{equation*}
	-f'(r;a)\sim\left[(p-q)U_{\infty,a}\right]^{1/(2-p)}r^{(\lambda_2-2)/(2-p)}=\left[(p-q)U_{\infty,a}\right]^{1/(2-p)}r^{-1/(q-p+1)}.
\end{equation*}
Comparing with~\eqref{deriv.f}, we find that $U_{\infty,a}$ is uniquely determined by
\begin{equation}\label{uinf}
U_{\infty,a}=\frac{1}{p-q}(\mu K^*)^{2-p}>0.
\end{equation}
Similarly, transforming back $Z(\cdot;a)$ to $f'(\cdot;a)$ according to~\eqref{change.var}, we infer from~\eqref{p17} that
\begin{equation*}
	Z(\ln{r};a)=r(-f'(r;a))^{q-p+1}\sim Z_*+V_{\infty,a}r^{\lambda_3}  \qquad {\rm as} \ r\to\infty,
\end{equation*}
hence
\begin{equation*}
	-r^{1/(q-p+1)}f'(r;a)\sim(Z_*+V_{\infty,a}r^{\lambda_3})^{1/(q-p+1)}\sim Z_*^{\mu+1}+\frac{V_{\infty,a}Z_*^\mu}{q-p+1}r^{\lambda_3}.
\end{equation*}
This gives, after one integration,
\begin{equation}\label{interm36}
f(r;a)\sim\frac{Z_*^{\mu+1}}{\mu}r^{-\mu}+\frac{V_{\infty,a}Z_*^\mu}{(\mu-\lambda_3)(q-p+1)}r^{\lambda_3-\mu} \qquad {\rm as} \ r\to\infty,
\end{equation}
which, together with Corollary~\ref{cor.B}, guarantees that $V_{\infty,a}\leq 0$.
Noticing that
\begin{equation*}
	Z_*^{\mu+1}=\mu K^*=[(p-q)U_{\infty,a}]^{1/(2-p)}, \qquad \lambda_3=-\theta
\end{equation*}
and setting
\begin{equation}
	A_a := -\frac{V_{\infty,a} Z_*^\mu}{(\mu-\lambda_3)(q-p+1)} \ge 0, \label{p18}
\end{equation}
we observe that~ \eqref{interm36} implies~\eqref{tail.SSS} with $A=A_a$.

Finally, assume for contradiction that $V_{\infty,a}=A_a=0$. Then $Z(\cdot;a)\equiv Z_*$ by Proposition~\ref{prop.uniq} and thus $r(-f'(r;a))^{q-p+1}=Z_*$ for $r\ge 0$; that is, $f(r;a)=Z_*^{\mu+1} r^{-\mu}/\mu = K^* r^{-\mu}$ for $r>0$ and a contradiction. Consequently, $A_a>0$ and the proof is complete.
\end{proof}

Another very useful consequence of the previous analysis is the following result.

\begin{corollary}\label{cor.uniq}
Let $(a_1,a_2)\in (0,\infty)^2$ be such that $a_1\ne a_2$ and assume that the corresponding solutions $f(\cdot;a_1)$ and $f(\cdot;a_2)$ to~\eqref{CP} satisfy $R(a_1)=R(a_2)=\infty$ and~\eqref{tail.SSS} with corresponding constants $(A_{a_1},A_{a_2})\in\real^2$. Then $A_{a_1}\neq A_{a_2}$.
\end{corollary}

\begin{proof}
Assume for contradiction that $A_{a_1}=A_{a_2}$. We then infer from~\eqref{uinf} and~\eqref{p18} that $(U_{\infty,a_1},V_{\infty,a_1}) = (U_{\infty,a_2},V_{\infty,a_2})$. Owing to Proposition~\ref{prop.uniq}, this equality implies that
\begin{equation*}
	 \big( X(\cdot;a_1),Y(\cdot;a_1),Z(\cdot;a_1)-Z_* \big) =  \big( X(\cdot;a_2),Y(\cdot;a_2),Z(\cdot;a_2)-Z_* \big),
\end{equation*}
hence $a_1=a_2$ by~\eqref{change.var}, and a contradiction.
\end{proof}

This corollary will be a decisive argument in the proof of the uniqueness part in Theorem~\ref{th.dim1}.

\section{Monotonicity and uniqueness}\label{sec.uniq}

Throughout this section, we restrict ourselves to dimension $N=1$. Let us recall that, for $a\in\mathcal{B}$, Theorem \ref{th.dimN} gives that
\begin{equation}\label{exp.infty}
f(r;a)=K^*r^{-\mu}\left[1-\frac{C(a)}{r}+o(r^{-1})\right] \qquad {\rm as} \ r\to\infty,
\end{equation}
for some $C(a)>0$. The next result establishes that two solutions of this kind are ordered.

\begin{proposition}\label{prop.monot}
Let $(a_1, a_2)\in\mathcal{B}^2$ such that $a_1<a_2$. Then $f(r;a_1)<f(r;a_2)$ for any $r\in[0,\infty)$.
\end{proposition}

\begin{proof}
We use the same approach as in \cite[Lemma~6]{FGV}. Since the proof therein is rather sketchy, we provide complete details below for the sake of completeness. Let $C_i=C(a_i)$ be the constant corresponding to $f(\cdot;a_i)$ in the asymptotic expansion~\eqref{exp.infty} and set $f_i:=f(\cdot;a_i)$ for $i\in\{1,2\}$ for simplicity, throughout this proof. According to the asymptotic expansion~\eqref{exp.infty}, there is $R_1>0$ such that, for any $r\geq R_1$, we have
$$
-\frac{K^*}{2}C_ir^{-(\mu+1)}\leq f_i(r)-K^*r^{-\mu}(1-C_ir^{-1})\leq \frac{K^*}{2}C_ir^{-(\mu+1)},
$$
or, equivalently,
\begin{equation}\label{interm17}
K^*r^{-\mu}\left(1-\frac{3C_i}{2r}\right)\leq f_i(r)\leq K^*r^{-\mu}\left(1-\frac{C_i}{2r}\right), \qquad r\geq R_1.
\end{equation}
We extend $f_i$ to $\real$ by setting $f_i(r):=f_i(-r)$ for $r\leq0$ and note that $f_i\in C^2(\real)$ for $i\in\{1,2\}$. Consider $\rho>0$ (to be determined later) and define
$$
h(r):=f_2(r-\rho), \qquad r\in\real.
$$
We first infer from~\eqref{interm17} that, for $r>R_1+\rho$, we have
\begin{equation*}
\begin{split}
h(r)-f_1(r)&=f_2(r-\rho)-f_1(r)\geq K^* \left[ (r-\rho)^{-\mu} \left(1-\frac{3C_2}{2(r-\rho)}\right) - r^{-\mu} \left(1-\frac{C_1}{2r}\right)\right]\\
&\geq K^* r^{-\mu} \left[ \left(1-\frac{\rho}{r}\right)^{-\mu} - \frac{3C_2}{2r} \left(1-\frac{\rho}{r}\right)^{-(\mu+1)} - 1 + \frac{C_1}{2r}\right]\\
&=K^*r^{-\mu}\left[\mu\int_{-\rho/r}^0(1+z)^{-(\mu+1)}\,dz+\frac{C_1}{2r}-\frac{3C_2}{2r}\left(1-\frac{\rho}{r}\right)^{-(\mu+1)}\right]\\
&\geq K^*r^{-\mu}\left[\frac{\mu \rho}{r} + \frac{C_1}{2r} - \frac{3C_2}{2r} \left(1-\frac{\rho}{r}\right)^{-(\mu+1)}\right]\\
&\geq K^* r^{-\mu} \left[\frac{2\mu\rho+C_1-3C_2}{2r} + \frac{3C_2}{2r} \left(1-\left(1-\frac{\rho}{r}\right)^{-(\mu+1)}\right)\right]\\
&=K^*r^{-\mu}\left[\frac{2\mu\rho + C_1-3C_2}{2r} - (\mu+1)\frac{3C_2}{2r}\int_{-\rho/r}^{0}(1+z)^{-(\mu+2)}\,dz\right]\\
&\geq K^* r^{-\mu} \left[\frac{2\mu\rho +C_1-3C_2}{2r} - (\mu+1)\frac{3C_2}{2r} \int_{-\rho/r}^{0}\left(1-\frac{\rho}{r}\right)^{-(\mu+2)}\,dz\right]\\
&\geq K^* r^{-\mu} \left[ \frac{2\mu\rho+C_1-3C_2}{2r} - (\mu+1)\frac{3\rho C_2}{2r^2} \left(1-\frac{\rho}{r}\right)^{-(\mu+2)}\right].
\end{split}
\end{equation*}
Assuming further that $r\geq k\rho$ in addition to $r\geq R_1+\rho$ for some $k\geq1$, we conclude that
\begin{equation}\label{interm18}
h(r)-f_1(r)\geq\frac{K^*}{2} r^{-(\mu+1)} \left[ 2\mu\rho +C_1- 3C_2 - \frac{3(\mu+1)C_2}{k} \left(\frac{k-1}{k}\right)^{-(\mu+2)}\right]
\end{equation}
for $r\geq\max\{R_1+\rho,k\rho\}$. We now turn our attention to $r<0$. If $r<-R_1$, by performing analogous steps as in the above estimate, we deduce from \eqref{interm17} that
\begin{equation*}
\begin{split}
h(r)-f_1(r)&=f_2(r-\rho)-f_1(r) =f_2(\rho-r) - f_1(r) \\
&\leq K^* \left[(\rho-r)^{-\mu} \left(1-\frac{C_2}{2(\rho-r)}\right) - (-r)^{-\mu} \left(1+\frac{3C_1}{2r}\right)\right]\\
&\leq K^* |r|^{-\mu} \left[\left(1+\frac{\rho}{|r|}\right)^{-\mu} - 1 + \frac{3C_1}{2|r|}\right]\\
&=K^* |r|^{-\mu} \left[ -\mu\int_{0}^{\rho/|r|}(1+z)^{-(\mu+1)}\,dz + \frac{3C_1}{2|r|}\right]\\
&\leq K^* |r|^{-\mu} \left[ \frac{-\mu\rho}{|r|} \left(1+\frac{\rho}{|r|}\right)^{-(\mu+1)} + \frac{3C_1}{2|r|}\right]\\
&\leq K^* |r|^{-\mu} \left[ \frac{3C_1-2\mu\rho}{2|r|} + \frac{\mu\rho}{|r|} \left(1-\left(1+\frac{\rho}{|r|}\right)^{-(\mu+1)}\right) \right]\\
&=K^* |r|^{-\mu} \left[ \frac{3C_1-2\mu\rho}{2|r|} + (\mu+1) \frac{\mu\rho}{|r|} \int_{0}^{\rho/|r|} (1+z)^{-(\mu+2)}\,dz\right]\\
&\leq K^* |r|^{-\mu} \left[ \frac{3C_1-2\mu\rho}{2|r|} + \frac{\mu(\mu+1)\rho^2}{r^2}\right].
\end{split}
\end{equation*}
Therefore, if $r\leq-k\rho$ in addition to $r\leq-R_1$, we conclude that
\begin{equation}\label{interm19}
h(r)-f_1(r)\leq\frac{K^*}{2} |r|^{-(\mu+1)} \left[ 3C_1-2\mu\rho + \frac{2\mu(\mu+1)\rho}{k} \right],
\end{equation}
for $|r|\geq\max\{R_1,k\rho\}$. We next choose $\rho$ such that
\begin{equation*}
	\rho>\max\left\{\frac{3C_1}{\mu},\frac{3C_2-C_1}{\mu},R_1\right\},
\end{equation*}
and pick $k\geq 2(\mu+1)$ sufficiently large such that
\begin{equation*}
	\frac{\mu \rho}{3C_2(\mu+1)}\geq\frac{1}{k}\left(\frac{k-1}{k}\right)^{-(\mu+2)}.
\end{equation*}
With this choice of $\rho$ and $k$, we note that $k\rho \geq R_1+\rho \geq R_1$ and we infer from~\eqref{interm18} and~\eqref{interm19} that $h(r)-f_1(r)>0$ for $r\geq k\rho$ and $h(r)-f_1(r)<0$ for $r\leq -k\rho$. We have thus shown that
\begin{equation}\label{interm20}
\begin{split}
&f_2(r-\rho)>f_1(r) \qquad {\rm for} \ r\geq k\rho,\\
&f_2(r-\rho)<f_1(r) \qquad {\rm for} \ r\leq-k\rho.
\end{split}
\end{equation}
Consider now $r\in(-\infty,0)$. There exists an integer $j\geq 0$ such that $-(j+1)\rho \leq r < -j\rho$. Either $j\geq k$, so that $r\leq-k\rho$, $r-(k+1)\rho\leq r-\rho\leq 0$ and it follows from the monotonicity of $f_2$ on $(-\infty,0)$ and~\eqref{interm20} that
$$
f_2(r-(k+1)\rho)\leq f_2(r-\rho)<f_1(r).
$$
Or $j\in\{0,1,...,k-1\}$. In that case, $r-(k-j)\rho=r+j\rho-k\rho<-k\rho$ and we infer from~\eqref{interm20} that
$$
f_2(r-(k-j)\rho-\rho)<f_1(r-(k-j)\rho),
$$
and the monotonicity of $f_1$ and $f_2$ on $(-\infty,0)$ entails that
$$
f_2(r-(k+1)\rho)\leq f_2(r-(k+1-j)\rho)<f_1(r-(k-j)\rho)\leq f_1(r).
$$
Putting the previous lines together, we have established that
\begin{equation}\label{interm21}
	f_2(r-(k+1)\rho)<f_1(r) \qquad {\rm for} \ r<0.
\end{equation}
Consider next $r>(k+1)\rho$. Then $r-\rho\geq r-(k+1)\rho\geq0$, and we derive from the monotonicity of $f_2$ on $(0,\infty)$ and~\eqref{interm20} that
\begin{equation}\label{interm22}
	f_2(r-(k+1)\rho)\geq f_2(r-\rho)>f_1(r) \qquad {\rm for} \ r>(k+1)\rho.
\end{equation}
It readily follows from~\eqref{interm21} and~\eqref{interm22} that the function
$$
\chi(r):=f_2(r-(k+1)\rho)-f_1(r), \qquad r\in\real,
$$
has at least one zero in the interval $[0,(k+1)\rho]$. Let us denote the smallest zero of $\chi$ by $r_0\in[0,(k+1)\rho]$. On the one hand, we have $\chi(r_0)=0$ and $\chi(r)<0$ for $r<r_0$. On the other hand, if $r\in(r_0,(k+1)\rho]$, then
$$
0>r-(k+1)\rho>r_0-(k+1)\rho,
$$
and the monotonicity of $f_1$ on $(0,\infty)$ and of $f_2$ on $(-\infty,0)$ give that
$$
f_2(r-(k+1)\rho)>f_2(r_0-(k+1)\rho)=f_1(r_0)>f_1(r),
$$
so that $\rho(r)>0$ for $r>r_0$. We have just proved that $h$ has a single zero $r_0$ in $[0,(k+1)\rho]$.

We next recall that the functions $u_1$ and $u_2$ defined by
$$
u_i(t,x) := (1-t)^{\alpha}f_i(x(1-t)^{\beta}), \qquad (t,x)\in[0,1]\times\real,
$$
are solutions to the partial differential equation~\eqref{eq1} for $i\in\{1,2\}$. Since~\eqref{eq1} is invariant with respect to translations, the function
$$
U_2(t,x) := u_2(t,x-(k+1)\rho), \qquad (t,x)\in[0,1]\times\real,
$$
also solves \eqref{eq1}. Noticing that we can write
$$
U_2(t,x)=(1-t)^{\alpha}f_2(x(1-t)^{\beta}-(1-t)^{\beta}(k+1)\rho),
$$
we observe on the one hand that, for $t\in(0,1)$,
\begin{equation}\label{interm23}
U_2(t,(k+1)\rho)=(1-t)^{\alpha} a_2 \geq(1-t)^{\alpha}a_1\geq u_1(t,(k+1)\rho),
\end{equation}
where we have used the fact that $a_1=f_1(0)=\max\{f_1(r):r\in\real\}$. On the other hand, for $x\in((k+1)\rho,\infty)$ we infer from~\eqref{interm22} that
\begin{equation}\label{interm24}
U_2(0,x)=f_2(x-(k+1)\rho)>f_1(x)=u_1(0,x).
\end{equation}
The comparison principle applied on $(0,1)\times((k+1)\rho,\infty)$, together with~\eqref{interm23} and~\eqref{interm24}, implies that $U_2(t,x)\geq u_1(t,x)$ for $(t,x)\in(0,1)\times((k+1)\rho,\infty)$, whence
$$
	f_2(x(1-t)^{\beta}-(k+1)\rho(1-t)^{\beta})\geq f_1(x(1-t)^{\beta}), \qquad (t,x)\in(0,1)\times((k+1)\rho,\infty),
$$
or equivalently
\begin{equation}\label{interm25}
f_2(r-(k+1)\rho(1-t)^{\beta})\geq f_1(r), \qquad t\in(0,1), \ r\in((k+1)\rho(1-t)^{\beta},\infty).
\end{equation}
Owing to the continuity of $f_1$ and $f_2$, we can pass to the limit as $t\to1$ in~\eqref{interm25} and conclude that $f_1(r)\leq f_2(r)$ for any $r>0$, the inequality being obvious for $r=0$ from the fact that $a_1=f_1(0)<f_2(0)=a_2$.

Let us finally assume for contradiction that there is $r_0\ge 0$ such that $f_1(r_0)=f_2(r_0)$. Then $r_0>0$ due to $a_1<a_2$ and it follows from the just established non-positivity of $f_1-f_2$ that $f_1-f_2$ has a maximum at $r_0$; that is, we have also $f_1'(r_0)=f_2'(r_0)$ and the Cauchy-Lipschitz theorem implies that $f_1\equiv f_2$ on $[0,\infty)$, contradicting $a_1<a_2$. The proof is now complete.
\end{proof}

We are now in a position to complete the proof of the uniqueness part in Theorem~\ref{th.dim1}.

\begin{proof}[Proof of Theorem~\ref{th.dim1}: uniqueness]
Assume for contradiction that there are $(a_1,a_2)\in\mathcal{B}^2$ with $a_1<a_2$. Setting $f_i=f(\cdot;a_i)$ for $i\in\{1,2\}$ as above, we infer from Proposition~\ref{prop.monot} that $f_1(r)<f_2(r)$ for any $r\geq 0$. Since
\begin{equation}
	C_i := \lim\limits_{r\to\infty} \frac{r\big( K^* - r^\mu f_i(r) \big)}{K^*} \label{p20}
\end{equation}
is well-defined and positive for $i\in\{1,2\}$ according to~\eqref{exp.infty} and Theorem~\ref{th.dimN}, we deduce from Proposition~\ref{prop.monot} that $0<C_2\leq C_1$. In fact, the strict inequality $C_2<C_1$ holds true due to Corollary~\ref{cor.uniq}. Fix then
$$
\tau_0=\frac{1}{2}\left[1-\left(\frac{C_2}{C_1}\right)^{1/\beta}\right]\in\left(0,\frac{1}{2}\right).
$$
We then readily find that $C_2<C_1(1-\tau_0)^{\beta}$, which guarantees the existence of $\varepsilon_0>0$ such that $C_2+\varepsilon_0 < (C_1-\varepsilon_0)(1-\tau_0)^{\beta}$. We then observe that, for $r>0$ and $\tau\in [0,\tau_0]$,
\begin{align*}
	K^* r^{-\mu} - K^* (C_1-\varepsilon_0) r^{-(\mu+1)} & \leq K^* r^{-\mu} - K^* (C_2+\varepsilon_0) (1-\tau_0)^{-\beta} r^{-(\mu+1)} \\
	& \leq K^* r^{-\mu} - K^* (C_2+\varepsilon_0) (1-\tau)^{-\beta} r^{-(\mu+1)} ,
\end{align*}
whence, taking into account that $\alpha=\mu\beta$,
\begin{equation}\label{interm37}
	\begin{split}
	K^*r^{-\mu} & - K^* (C_1-\varepsilon_0) r^{-(\mu+1)} \\
	& \leq (1-\tau)^\alpha \left[ K^* (1-\tau)^{-\mu\beta} r^{-\mu} - K^* (C_2+\varepsilon_0) (1-\tau)^{-(\mu+1)\beta} r^{-(\mu+1)} \right].
	\end{split}
\end{equation}
Owing to~\eqref{p20}, there is $R_{\varepsilon_0}>0$ such that, for $r\ge R_{\varepsilon_0}$,
\begin{align*}
	f_1(r) & \le K^* r^{-\mu} - K^*(C_1-\varepsilon_0) r^{-(\mu+1)}, \\
	f_2(r) & \ge K^* r^{-\mu} - K^*(C_2+\varepsilon_0) r^{-(\mu+1)} .
\end{align*}
Now, for $r\ge R_0 := R_{\varepsilon_0}(1-\tau_0)^{-\beta}$ and $\tau\in [0,\tau_0]$, one has $r\ge r(1-\tau)^\beta \ge R_{\varepsilon_0}$ and we infer from~\eqref{interm37} and the above bounds that
\begin{equation*}\
	f_1(r)\le(1-\tau)^{\alpha}f_2(r(1-\tau)^{\beta}), \qquad (\tau,r)\in [0,\tau_0]\times [R_0,\infty).
\end{equation*}
It remains to work on the compact interval $r\in[0,R_0]$. But, since $f_1(r)<f_2(r)$ for $r\in[0,R_0]$ and
\begin{equation*}
	\lim\limits_{\tau\to0}(1-\tau)^{\alpha}f_2(r(1-\tau)^{\beta})=f_2(r),
\end{equation*}
uniformly for $r\in[0,R]$, it follows that there exists $\tau_1\in(0,\tau_0)$ such that
\begin{equation*}
	f_1(r) \le(1-\tau)^{\alpha} f_2(r(1-\tau)^{\beta}), \qquad (\tau,r)\in [0,\tau_1]\times [0,R_0].
\end{equation*}
Consequently,
\begin{equation}\label{interm38}
	f_1(r) \le(1-\tau)^{\alpha} f_2(r(1-\tau)^{\beta}), \qquad (\tau,r)\in [0,\tau_1]\times [0,\infty),
\end{equation}
so that, coming back to  self-similar solutions and defining the functions
$$
U_1(t,x)=(1-t)^{\alpha}f_1(|x|(1-t)^{\beta}), \qquad U_2(t,x)=(1-\tau_1-t)^{\alpha}f_2(|x|(1-\tau_1-t)^{\beta}),
$$
both are solutions to Eq.~\eqref{eq1} and~\eqref{interm38} implies that $U_1(0,x)<U_2(0,x)$ for any $x\in\real$. The comparison principle then entails that
$$
U_1(t,x)\leq U_2(t,x), \qquad t\geq 0.
$$
But this is a contradiction, since $U_2$ vanishes uniformly at time $1-\tau_1<1$, while the extinction time of $U_1$ is $t=1$. This contradiction completes the proof.
\end{proof}

\section{Discussion}\label{sec.disc}

Having established in Theorem~\ref{th.dim1} the existence and uniqueness of a self-similar solution to~\eqref{eq1} in one space dimension, we expect that it attracts a wide class of non-negative solutions to the associated initial value problem near their extinction time. One possible approach to prove such a convergence result is to construct a Lyapunov functional, which is in principle possible in one space dimension, following the approach designed in \cite{Z68}. Formal computations that we have performed indicate that a Lyapunov functional is indeed available in this case, but a rigorous justification requires  approximation arguments and uniform estimates that we have been unable to derive. In particular, the approach described in \cite{Z68} requires a detailed study of the system of ordinary differential equations
\begin{equation}
	V'(y) = W(y), \quad W'(y) = |W(y)|^{2-p} \big( |W(y)|^q - \alpha V(y) - \beta y W(y) \big), \qquad y\in\real. \label{ode}
\end{equation}
However, solutions to~\eqref{ode} might blow up at a finite $y\in\real$ due to the superlinearity of the right-hand side of the $W$-equation with respect to $W$ and could also be unbounded, increasing as $e^{Cy^2}$, as a consequence of the positivity of $\beta$ and the linear dependence of the right-hand side of the $W$-equation with respect to $y$. Despite some attempts, we have yet been unable to design a suitable approximation scheme to justify the availability of a Lyapunov functional and we hope to return to that problem in the future.

\section*{Acknowledgements} This work is partially supported by the Spanish project PID2020-115273GB-I00 and by the Grant RED2022-134301-T (Spain). Part of this work has been developed during visits of R. G. I. to Institut de Math\'ematiques de Toulouse and to Laboratoire de Math\'ematiques LAMA, Universit\'e de Savoie, and of Ph. L. to Universidad de Valencia, Instituto de Ciencias Matem\'aticas de Madrid (ICMAT) and Universidad Rey Juan Carlos, and both authors thank these institutions for hospitality and support.

\bibliographystyle{plain}

\end{document}